\documentclass[12pt]{amsart}

\usepackage[matrix,arrow,curve,frame]{xy}
\usepackage{amsmath,amsthm,amssymb,enumerate,stmaryrd}
\usepackage{latexsym}
\usepackage{amscd}
\usepackage[colorlinks=true]{hyperref}
\usepackage{euscript}
\usepackage{url}
\usepackage{accents}
\usepackage{mathrsfs}
\usepackage{hyperref}
\usepackage{tikz}
\usepackage{braids}
\usepackage{pgfplots}

\newtheorem*{Thm}{Theorem}
\newtheorem{thm}[subsection]{Theorem}
\newtheorem{defn}[subsection]{Definition}
\newtheorem*{Defi}{Definition}
\newtheorem*{Remark}{Remark}

\newtheorem{claim}[subsection]{Claim}
\newtheorem{corr}[subsection]{Corollary}
\newtheorem{lemma}[subsection]{Lemma}

\newtheorem{remark}[subsection]{Remark}

\newtheorem{Conv}[subsection]{Convention}

\theoremstyle{definition}
\newtheorem{example}[subsection]{Example}

\newcommand{\cat}{\mathcal}
\newcommand{\lra}{\longrightarrow}

\newcommand{\llra}[1]{\stackrel{#1}{\lra}}

\newcommand{\R}{\mathbb R}
\newcommand{\Z}{\mathbb Z}
\newcommand{\F}{\mathbb F}

\newcommand{\C}{\mathbb C}

\newcommand{\I}{{\cat I}}
\newcommand{\Jj}{{\cat J}}

\newcommand{\sBC}{s\mathscr{B}}
\newcommand{\sBCU}{s\mathscr{B}_\infty}
\newcommand{\sNBC}{s\mathscr{B}}
\newcommand{\BC}{\mathscr{B}}

\newcommand{\sBS}{s\cat B\cat Sh}
\newcommand{\BS}{\cat B\cat Sh}
\newcommand{\BSa}{\cat Bt\cat S}
\newcommand{\sBSa}{s\cat Bt\cat S}
\newcommand{\Sh}{\cat Sh}

\DeclareMathOperator{\Map}{Map}

\DeclareMathOperator{\Br}{Br}
\DeclareMathOperator{\GL}{GL}

\DeclareMathOperator{\Gr}{Gr}

\DeclareMathOperator{\E}{E}
\DeclareMathOperator{\Ho}{\mbox{H}}

\DeclareMathOperator{\SU}{U}

\newfont{\german}{eufm10}

  \DeclareMathOperator{\hocolim}{hocolim}

\newcommand\qu{/\kern-.7ex/}

\begin{document}
\pagestyle{plain}

\title
{Symmetry Breaking and Link homologies I }
\author{Nitu Kitchloo}
\address{Department of Mathematics, Johns Hopkins University, Baltimore, USA}
\email{nitu@math.jhu.edu}
\thanks{The author acknowledges support by the Simons Foundation and the Max Planck Institute for Math. during the development of these ideas.}
\maketitle

\begin{abstract}
Given a compact connected Lie group $G$ endowed with root datum, and an element $w$ in the corresponding Artin braid group for $G$, we describe a filtered $G$-equivariant stable homotopy type, up to a notion of quasi-equivalence. We call this homotopy type {\em Strict Broken Symmetries}, $\sBC(w)$. As the name suggests, $\sBC(w)$ is constructed from the stack of principal $G$-connections on a circle, whose symmetry is broken between consecutive arcs in a manner prescribed by a  presentation of $w$. We show that $\sBC(w)$ is independent of the choice of presentation of $w$, and also satisfies Markov type properties. Specializing to the case of the unitary group $G=\SU(r)$, these properties imply that $\sBC(w)$ is an invariant of the link $L$ obtained by closing the $r$-stranded braid $w$. As such, we denote it by $\sBC(L)$. In the follow up to this article \cite{Ki2}, we will show that the construction of strict broken symmetries allows us to incorporate twistings. Under suitable conditions, $\SU(r)$-equivariant (twisted) cohomology theories $\E_{\SU(r)}$ applied to $\sBC(L)$ give rise to a spectral sequence of link invariants converging to $\E_{\SU(r)}^\ast(\sBCU(L))$, where $\sBCU(L)$ is the direct limit of the filtration of $\sBC(L)$. For instance, in \cite{Ki2} we describe Triply-graded link homology as the $E_2$-term in the spectral sequence one obtains on applying Borel-equivariant singular cohomology $\Ho_{\SU(r)}$ to $\sBC(L)$. We also study a universal twist of $\Ho_{\SU(r)}$. Here one recovers $sl(n)$-link homologies for any value of $n$ (depending on the choice of specialization of the universal twist).  
\end{abstract}

\tableofcontents

\section{Introduction}

\noindent
The main result of this article is the construction of a filtered $\SU(r)$-equivariant stable homotopy type $\sBC(L)$ for links $L$ that can be described as the closure of $r$-stranded braids, namely, elements of the braid group $\Br(r)$. We call this spectrum the spectrum of {\em strict broken symmetries} because it is built from the stack of principal $\SU(r)$-connections on a circle endowed with reductions of the structure group to the maximal torus at various points on the circle, and prescribed holonomy subgroups along the interpolating arcs. We will elaborate on this point presently. Even though we have invoked the category of equivariant spectra, for links $L$ that can be expressed as the closure of a positive braid, our spectrum $\sBC(L)$ can be described entirely by the geometry of an underlying $\SU(r)$-equivariant {\em space} of strict broken symmetries. 

\medskip
\noindent
For the convenience of non-experts, all the results in the introduction will be formulated for links given by the closure of a positive braid, with the general result for arbitrary braids described in later sections. We also point out that several results in this article will be shown to hold for arbitrary compact connected Lie groups $G$. We have chosen to highlight the case $G = \SU(r)$ in the introduction because of the natural interpretation of our results in terms of link homology. Broadly speaking, the interpretation of our homotopy type in terms of Khovanov-Rozansky homology as constructed in \cite{J2} can be described as follows. Details can be found in \cite{Ki2}. 

\medskip
\noindent
Given a link $L$ that is expressed by the closure of a braid diagram, Khovanov-Rozansky homology of $L$ is traditionally described by first constructing a triply-graded complex $\mathscr{C}(L)$ with three commuting differentials $d_+, d_{-}, d_v$. The two differentials $d_+$ and $d_-$ are known as the ``Matrix Factorization" differentials, and the remaining differential $d_v$ is known as the ``cubical" or "resolution of crossings" differential. With the above setup, the Khovanov-Rozansky triply graded homology of the link $L$ is defined as the successive homology $\Ho(\Ho(\mathscr{C}(L), d_+), d_v)$. Similarly, the Khovanov-Rozansky $sl(n)$-link homology of $L$ is defined as the successive homology $\Ho(\Ho(\mathscr{C}(L), d_++d_{-}), d_v)$, where the definition of the differential $d_-$ depends on $n$. 

\medskip
\noindent
Ignoring gradings for now, we may recover the above algebraic framework using our homotopy type $\sBC(L)$ as follows. The filtration on $\sBC(L)$ gives rise to an associated graded equivariant homotopy type $\Gr(\sBC(L))$ (see section \ref{Filteredalgebra}). In \cite{Ki2} we prove that the equivariant singular cohomology of $\Gr(\sBC(L))$ is isomorphic to $\Ho(\mathscr{C}(L), d_+)$. Furthermore, by virtue of being an associated graded object, $\Gr(\sBC(L))$ admits a differential $\partial$ in the homotopy category. The induced differential $\partial^\ast$ on $\Ho^\ast_{\SU(r)}(\Gr(\sBC(L))$ can be identified with $d_v$. Hence we recover the Khovanov-Rozansky triply graded link homology of $L$ entirely in terms of the equivariant singular cohomology of the filtered object $\sBC(L)$. 

\medskip
\noindent
A final piece of topological structure (see below) that we describe in detail in \cite{Ki2} is a ``derived local system" on $\sBC(L)$ which allows us to define a version of \textit{twisted} equivariant singular cohomology. This twisting gives rise to a differential on the equivariant cohomology of $\sBC(L)$ which we conjecture in \cite{Ki2} to be the differential $d_-$ in terms of the above identification. This conjecture has recently been proven by T. Mej\'{i}a Gomez \cite{Go}, thereby showing how our homotopy type $\sBC(L)$ can also be used to recover Khovanov-Rozansky $sl(n)$-link homology in terms of the twisted equivariant cohomology of $\sBC(L)$. 

\medskip
\noindent
The above discussion hopefully motivates why one is justified in thinking of $\sBC(L)$ as an equivariant stable homotopy type for links. Even though we have described the evaluation of $\sBC(L)$ on equivariant singular cohomology, $\sBC(L)$ can be evaluated under arbitrary equivariant cohomology theories potentially giving rise to other interesting link homologies. In addition, the derived local system alluded to above is well know to twist many interesting equivariant cohomology theories (for instance equivariant K-theory). This is discussed further in \cite{Ki2} and remains an active direction of our research.

\medskip
\noindent
Before we proceed, let us say a few words about the category of $G$-spectra that will be used in this article. The main results of this article can be understood with very little background on equivariant spectra. It is helpful to bear in mind that $G$-spectra may be seen as a natural localization of the category of $G$-spaces where one is allowed to desuspend by arbitrary finite dimensional $G$-representations. As with $G$-spaces, one may evaluate $G$-spectra on $G$-equivariant cohomology theories. Given a subgroup $H < G$, one has restriction and induction functors defined respectively by considering a $G$-spectrum as an $H$-spectrum, or by inducing up an $H$-spectrum $X$ to the $G$-spectrum ${G_+} \wedge_H X$. As one would expect, the induction from $H$-spectra to $G$-spectra is left adjoint to restriction. For those somewhat familiar with the language, by an equivariant $G$-spectrum we mean an equivariant spectrum indexed on a complete $G$-universe \cite{LMS}. 

\medskip
\noindent
The spectra we study in this article are filtered by a finite increasing filtration $F_t X$. The associated graded object $\Gr_t(X)$ of such a spectrum has a natural structure of a chain complex in the homotopy category of $G$-spectra. In particular, one may define an {\em acyclic filtered $G$-spectrum} $X$ so that the associated graded object $\Gr_r(X)$ admits stable null homotopies. The notion of acyclicity allows us to define a notion of {\em quasi-equivalence} on our category of filtered $G$-spectra by demanding that two filtered $G$-spectra are equivalent if they are connected by a zig-zag of maps each of whose fiber (or cofiber) is acyclic. 

\medskip
\noindent
Returning to the main application of this article, we show that a braid $w$ on $r$-strands gives rise to a filtered equivariant $\SU(r)$-spectrum of strict broken symmetries, denoted by $\sBC(w)$, which is well defined up to quasi-equivalence. Before we get to the definition of strict broken symmetries, let us first offer a geometric description of the $\SU(r)$-spectrum of broken symmetries. Consider a braid element $w \in \Br(r)$, where $\Br(r)$ stands for the braid group on $r$-strands. For the sake of exposition, consider the case of a positive braid, i.e. one that can be expressed in terms of positive exponents of the elementary braids $\sigma_i$ for $i < r$. 
Let $I = \{ i_1, i_2, \ldots, i_k \}$ denote an indexing sequence with $i_j < r$, so that a positive braid $w$ admits a presentation in terms of the fundamental generators of the braid group $\Br(r)$, $w = w_I := \sigma_{i_1} \sigma_{i_1} \ldots \sigma_{i_k}$. Let $T$, or $T^r$ (if we need to specify rank), be the standard maximal torus, and let $G_i$ denote the unitary (block-diagonal) form in the reductive Levi subgroup having roots $\pm \alpha_i$. We consider $G_i$ as a two-sided $T$-space under the left (resp.\,right) multiplication. 

\medskip
\noindent
The equivariant $\SU(r)$-spectrum of broken symmetries is defined as the (suspension) spectrum corresponding to the $\SU(r)$-space $\BC(w_I)$ that is induced up from a $T$-space $\BC_T(w_I)$
\[ \BC(w_I) := \SU(r) \times_T ({G_{i_1}} \times_T {G_{i_2}} \times_T \cdots \times_T {G_{i_k}}) = \SU(r) \times_T \BC_T(w_I) , \]
with the $T$-action on $\BC_T(w_I) := (G_{i_1} \times_T G_{i_2} \times_T \cdots \times_T G_{i_k})$ given by endpoint conjugation
\[ t \, [(g_1, g_2, \cdots, g_{k-1}, g_k)] := [(tg_1, g_2, \cdots, g_{k-1}, g_kt^{-1})]. \]

\noindent
As mentioned above, the $\SU(r)$-stack $\SU(r) \times_T (G_{i_1} \times_T G_{i_2} \times_T \cdots \times_T G_{i_k})$ is equivalent to the stack of $\SU(r)$-connections on a circle with $k$ marked points, with the structure group being reduced to $T$ at the points, and the holonomy in $G_i$ along the $i$-th arc. 

\medskip
\noindent
Let us now describe the derived local system that allows for a twisting of our framework. Notice that each space of broken symmetries $\BC(w_I)$ admits a canonical map (given by composing the holonomies along the arcs) to the stack of principal connections on a circle, which is equivalent to the adjoint action of $\SU(r)$-action on itself
\[ \rho_I : \BC(w_I) \longrightarrow \SU(r), \quad \quad [(g, g_{i_1}, \ldots, g_{i_k})] \longmapsto g (g_{i_1} \ldots g_{i_k}) g^{-1}. \]
These maps $\rho_I$ are clearly compatible under inclusions $J \subseteq I$. In particular, the spectra $\sBC(L)$ can be endowed with a $\SU(r)$-equivariant derived local system by pulling back any $\SU(r)$-equivariant derived local systems on $\SU(r)$. We will use this structure in \cite{Ki2} to twist the equivariant cohomology theories $\E_{\SU(r)}$ considered above and relate it to $sl(n)$-link homology. 

\medskip
\begin{Defi} (Strict broken symmetries and their normalization) 

\noindent
Let $L$ denote a link described by the closure of a positive braid $w \in \Br(r)$ with $r$-strands, and let $w_I$ be a presentation of $w$ as $w = \sigma_{i_1} \ldots \sigma_{i_k}$. We first define the limiting $\SU(r)$-spectrum $\sBCU(w_{I})$ of strict broken symmetries as the space that fits into a cofiber sequence of $\SU(r)$-spaces: 
\[ \hocolim_{J \in \I} \BC(w_{J}) \longrightarrow \BC(w_I) \longrightarrow \sBCU(w_{I}). \]
where $\I$ is the category of all proper subsets of $I = \{ i_1, i_2, \ldots, i_k \}$. 

\medskip
\noindent
The spectrum $\sBCU(w_I)$ admits a natural increasing filtration by spaces $F_t \, \sBC(w_I)$ defined as the cofiber on restricting the above homotopy colimit to the full subcategories $\I^t \subseteq \I$ generated by subsets of cardinality at least $(k-t)$, so that the lowest filtration is given by $F_0 \sBC(w_I) = \BC(w_I)$. 

\medskip
\noindent
Define the spectrum of strict broken symmetries $\sBC(w_I)$ to be the filtered spectrum $F_t \, \sBC(w_I)$ above. The normalized spectrum of strict broken symmetries of the link $L$ is defined as 
\[ \sNBC(L) := \Sigma^{-2k} \sBC(w_I). \]
\end{Defi}

\medskip
\noindent
In order for the normalized definition to make sense, one would require proving that the construction of $\sBC(L)$ is independent (up to quasi-equivalence) of the braid presentation $w_I$ used to describe $L$. This comes down to checking the braid group relations, and the first and second Markov property. The first Markov property and the braid group relations are established in sections \ref{Markov 1} and \ref{Braid}-\ref{Inverse} resp. Results of these sections in fact admit a generalization to any compact connected Lie group $G$, and we work with that generality in the first six sections. However, as mentioned earlier, we have chosen to highlight the case $G = \SU(r)$ in this introduction. 

\medskip
\noindent
 The second Markov property imposes a stability condition on the construction, requiring that it be invariant under the augmentation of $w$ by the elementary braid $\sigma_r$ (or its inverse) so as to be seen as a braid in $\Br(r+1)$. This is equivalent to the observation that the link $L$ is unchanged on adding an extra strand that is braided with the previous one. In proving invariance under the second Markov property, we encounter a subtle point. Notice that $\sBC(L)$ is induced up from a $T^r$-spectrum we shall denote by $\sBC_{T^r}(L)$. Proving invariance under the second Markov property would therefore require showing that the $\SU(r+1)$-spectrum obtained by considering $L$ as the closure of $w\sigma_r^{\pm}$ is induced from $\sBC_{T^r}(L)$ along the standard inclusion $T^r  < \SU(r+1)$. This requirement is {\em almost true} but for a small subtlety. We show in section \ref{Markov2} that when $L$ is seen as the closure of the $(r+1)$-stranded braid $w\sigma_r^{\pm}$, the corresponding $\SU(r+1)$-spectrum, $\sBC(L)$ is induced up from $\sBC_{T^r}(L)$ along a {\em different inclusion} $\Delta_r : T^r \longrightarrow \SU(r+1)$. This inclusion differs from the standard inclusion in the last entry. This observation is resolved in section \ref{Markov2} by observing that $\sBC(L)$ can be canonically enhanced to a {\em framed} link invariant up to a notion of quasi-equivalence\footnote{we will keep the reference to framing implicit in this article for the sake of simplicity}. 
 
\medskip
\begin{Thm} 
As a function of (framed) links $L$, the filtered $\SU(r)$-spectrum of strict broken symmetries $\sBC(L)$ is well-defined up to quasi-equivalence. In particular, the limiting equivariant stable homotopy type $\sBCU(L)$ is a well-defined (framed) link invariant in $\SU(r)$-equivariant spectra (see remark \ref{limiting}). We discuss $\sBCU(L)$ below. 
\end{Thm}

\smallskip
\noindent
In section \ref{one} we show that the construction of the link invariant $\sBC(L)$ admits an internal symmetry given by the Galois action by complex conjugation on all Levi subgroups. 

\smallskip
\noindent
An obvious way to obtain (group valued) link invariants from the filtered homotopy type $\sBC(L)$ is to apply an equivariant cohomology theory and invoke the filtration to set up a spectral sequence. Let $\E_G$ denote a family of equivariant cohomology theories indexed by the collection $G = \SU(r)$, with $r \geq 1$, endowed with natural compatiblity under restrictions 
\[ \Delta_j : \E_{\SU(r_1)} \wedge \E_{\SU(r_k)} \ldots \wedge \E_{\SU(r_m)} \llra{\cong} j^\ast \E_{\SU(m)}, \, \, \mbox{where} \, \, \,  j : \SU(r_1) \times \SU(r_2) \times \ldots \times \SU(r_k) \hookrightarrow \SU(m), \]
is an arbitrary injection of Lie groups. Given a family of equivariant cohomology theories $\E_{\SU(r)}$ as above that satisfies additional conditions described in definition \ref{condM2}, the filtration of $\sBC(L)$ described above does indeed give rise to a spectral sequence that converges to $\E^*_{\SU(r)}(\sBCU(L))$. The $E_2$-term of this spectral sequence is itself a link invariant, and is given by the cohomology of the associated graded complex for the filtration of $\sBC(L)$. We have 

\medskip
\begin{Thm} 
Assume that $\E_{\SU(r)}$ is a family of $\SU(r)$-equivariant cohomology theories as above that satisfy the conditions of definition \ref{condM2}. Then, given a link $L$ described as a closure of a positive braid presentation $w_I$ on $r$-strands, one has a spectral sequence converging to $\E^\ast_{\SU(r)}(\sBCU(L))$ and with $E_1$-term given by 
\[ E_1^{t,s} = \bigoplus_{J \in \I^t/\I^{t-1}} \E_{\SU(r)}^s(\BC(w_J)) \, \, \Rightarrow \, \, \E_{\SU(r)}^{s+t-2k}(\sBCU(L)). \]
The differential $d_1$ is the canonical simplicial differential induced by the functor described in definition \ref{poset}. In addition, the terms $E_q(L)$ are invariants of the link $L$ for all $q \geq 2$, upto an indeterminacy given by an overall shift in bi-degree. 
\end{Thm}

\smallskip
\noindent
The Galois symmetry mentioned above descends to an induced Galois symmetry of the link invariants $E_q(L)$ for $q \geq 2$. 

\medskip
\noindent
In \cite{Ki2}, we will relate special cases of the above spectral sequence to various well-known link homology theories. 
The limiting spectrum $\sBCU(L)$ can actually be described explicitly, and so we know exactly what the above spectral sequence converges to. More precisely, in section \ref{one} we will prove a generalization of the following theorem for arbitrary compact connected Lie groups $G$, and for braid words that are not necessarily positive.

\medskip
\begin{Thm}
Let $I = \{i_1, \ldots, i_k \}$ be an indexing set so that $w_I = \sigma_{i_1} \sigma_{i_2} \ldots \sigma_{i_k}$ is a braid word that closes to the link $L$. Let $V_I$ denote the representation of $T$ given by a sum of root spaces
\[ V_I = \sum_{j \leq k} w_{I_{j-1}}(\alpha_{i_j}) , \quad \mbox{where} \quad w_{I_{j-1}} = \sigma_{i_1} \ldots \sigma_{i_{j-1}}, \quad w_{I_0} = id, \]
and where $w_{I_{j-1}} (\alpha_{i_j})$ denotes the root space for the root given by the $w_{I_{j-1}}$-translate of the simple root $\alpha_{i_j}$. Then the $\SU(r)$-equivariant homotopy type of $\sBCU(L)$ is given by the equivariant Thom space (suitably desuspended)
\[ \sBCU(L) = \Sigma^{-2k} \SU(r)_+ \wedge_T (S^{V_I} \wedge T(w)_+), \]
where $S^{V_I}$ denotes the one-point compactification of the $T$-representation $V_I$, and $T(w)$ denotes the twisted conjugation action of $T$ on itself
\[ t (\lambda) := w^{-1} t w t^{-1} \,  \lambda \quad \mbox{where} \quad w = \sigma_{i_1} \ldots \sigma_{i_k}, \quad t \in T, \quad \lambda \in T(w). \]
\end{Thm}

\smallskip
\begin{Remark}
Note that the structure group $T$ of the above Thom spectrum can be reduced to the sub torus $T^w \subseteq T$ that is fixed by the Weyl element $w$ that underlies $w_I$. The torus $T^w$ is isomorphic to a product of rank one tori indexed by the components of $L$. More precisely, the factor corresponding to a particular component of $L$ is the diagonal in the standard sub torus of $T^r$ indexed by the strands belonging to that particular compoment. Since the cohomology of $\sBCU(L)$ (assuming Thom isomorphism) depends only on the number of components of $L$, we may think of $\sBCU(L)$ as a stable lift of the invariant $E_\infty(-1)$ (see Theorem 3 in \cite{J2}). Studying the rich internal structure of the spectrum $\sBCU(L)$ is work in progress. If $L$ has a presentation as the closure of a positive braid, then $\sBCU(L)$ appears to be closely related to the Braid Varieties studied in \cite{CGGS}. 
\end{Remark}

\medskip
\noindent
The entire topological construction presented in this article can be made in the context of a semiclassical topological $G$-gauge theory (see \cite{Ki3} for a definition). The space $\BC(w_I)$ of broken symmetries may be identified with gauge fields that decorate domain walls on a 2d-torus. These domain walls correspond to the breaking of gauge symmetry along the longitudinal circle from a compact minimal Levi subgroup $G_i \subseteq G$, to the maximal torus $T \subset G_i$. These walls are placed at points prescribed by the sequence $I$ along the longitudinal circle of the torus, and correspond to those $G$-connections along the meridian circle, whose holonomy-stabilizer subgroup in $G$ jumps from some conjugate of $T$ to a conjugate of $G_i$ around those points for $i \in I$. 

\medskip
\noindent
The moduli space $\BC(w_I)$ of gauge fields that decorate domain walls admit the loci $\BC(w_J)$ for $J \subsetneq I$. This loci corresponds to the trivial or the so called ``transparent defect" along some longitudinal arc in the sequence. Therefore, to strictly recover the defects prescribed by $I$, one needs to factor out the fields $\BC(w_J)$ for $J \subsetneq I$. This factoring out process, systematically implemented by the homotopy colimit, gives rise to a filtered equivariant homotopy type of gauge fields along domain walls: the \textit{strict} broken symmetries $\sBC(w_I)$. 

\medskip
\noindent
Now, picking an invertible boundary condition (with a framing) for our semiclassical topological $G$-gauge theory compactified on a circle, and tracing it over the space of strict broken symmetries, gives rise to a cosimplicial higher categorical object which we call the ``Khovanov-Rozansky" object. Under suitable conditions, this object is an invariant of the set of conjugacy classes in the Artin Braid group of $G$. In particular, when $G = \SU(n)$ one obtains link invariants which are categorical versions of Khovanov-Rozansky homology. We will outline this physical picture in the followup to this article \cite{Ki3}.

\medskip
\noindent
We are grateful to Mikhail Khovanov and Lev Rozansky for their constant interest. We thank Tudor Dimofte, David Gepner, Eugene Gorsky, Matt Hogancamp, David Jordan, Jack Morava, Dale  Rolfsen, Jacob Lurie and Paul Wedrich for helpful discussions. We also thank Vitaly Lorman, Tom\'as Mej\'{i}a Gomez, Apurva Nakade and Valentin Zakharevich for a thorough reading, especially of section \ref{Braid}. Finally, we thank Aaron Mazel-Gee for bringing to our attention an error in an earlier version of the proof of theorem \ref{braid rel}. 

\medskip
\noindent
Before we begin, we would like to mention a word about how we have organized the various definitions in this article, and those in its sequel \cite{Ki2}. All definitions have been labeled by the content they seek to define. The body of the definition also contains some basic observations about the definition that are straightforward to see, and don't rise to the level of claims. We hope that this helps the reader navigate the host of unavoidable definitions in this article.

\medskip
\section{(Strict) broken symmetries and the $G$-equivariant homotopy type} \label{two}

\medskip
\noindent
Let $G$ be a compact connected Lie group of rank $r$, and semisimple rank $r_s \leq r$.\footnote{All results in the first six sections of this paper hold for $G$ being the unitary form of an arbitrary  symmetrizable Kac-Moody group.} 
 Let us fix a root datum, and let $T \subset G$ denote the maximal torus acted upon by the Weyl group $W$. Let $\sigma_i \in W$ for $i \leq r_s$, be the generators corresponding to the simple roots $\alpha_i$, where $r_s \leq r$ is the semisimple rank of $G$. The Artin braid group for this root datum is a lift of the Coxeter presentation of $W$, and is defined by
\[ \Br(G) = \langle \, \sigma_i, \, 1 \leq i \leq r_s, \quad \sigma_i \sigma_j \sigma_i \ldots m_{ij} \, \mbox{terms} = \sigma_j \sigma_i \sigma_j \ldots m_{ij} \, \mbox{terms}  \rangle, \]
where the integers $m_{ij}$ are determined by the entries of the Cartan matrix corresponding to the root datum. In the special case of $G = \SU(r)$, the semisimple rank $r_s$ is $r-1$ and the braid group is the classical braid group $\Br(r)$ of braids with $r$ strands.

\medskip
\noindent
In this section we will construct a $G$-equivariant filtered stable homotopy type called the {\em strict broken symmetries}. All our $G$-spectra are genuine, by which we mean they are indexed on the complete $G$-universe. We first begin by defining the equivariant $G$-spectrum of broken symmetries given a presentation of an element $w \in \Br(G)$. Broken symmetries will then be assembled to construct strict broken symmetries. We begin with the definitions for positive braids. 

\medskip
\begin{defn} \label{BC+} (Broken symmetries: Positive braids)

\noindent
Let $I = \{ i_1, i_2, \ldots, i_k \}$ denote an indexing sequence with $i_j \leq r_s$, so that a positive braid $w$ admits a presentation in terms of the fundamental generators of $\Br(G)$, $w = w_I := \sigma_{i_1} \sigma_{i_1} \ldots \sigma_{i_k}$. Let $G_i$ denote the unitary form in the reductive Levi subgroup having roots $\pm \alpha_i$. We consider $G_i$ as a two-sided $T$-space under the canonical left(resp. right) multiplication. 

\smallskip
\noindent
Define the equivariant $G$-spectrum of broken symmetries to be the suspension spectrum of 
\[ \BC(w_I) := G \times_T ({G_{i_1}} \times_T {G_{i_2}} \times_T \cdots \times_T {G_{i_k}}) = G \times_T \BC_T(w_I) , \]
where the $T$-action on $\BC_T(w_I) := (G_{i_1} \times_T G_{i_2} \times_T \cdots \times_T G_{i_k})$ is given by endpoint conjugation
\[ t \,  [(g_1, g_2, \cdots, g_{k-1}, g_k)] := [(tg_1, g_2, \cdots, g_{k-1}, g_kt^{-1})]. \]
\end{defn}

\medskip
\begin{remark}
As was already mentioned in the introduction, let us point out again that the $G$-stack $G \times_T (G_{i_1} \times_T G_{i_2} \times_T \cdots \times_T G_{i_k})$ is equivalent to the stack of principal $G$-connections on the trivial $G$-bundle over $S^1$, endowed with a reduction of the structure group to $T$ at $k$ distinct marked points, and with the property that the holonomy along the $i$-th arc belongs to the subgroup $G_i$ in terms of this reduction. The stack of strict broken symmetries, which we will encounter later, should be interpreted as the stack obtained from broken symmetries by factoring out those connections whose holonomy preserves the $T$-structure in some arc. 
\end{remark}

\medskip
\noindent
Let us now address the matter of braid elements $w$ that admit a presentation with negative exponents. Let $\zeta_i$ denote the virtual $G_i$ representation $(\mathfrak{g}_i - r\R)$, where $\mathfrak{g}_i$ is the adjoint representation of $G_i$, and $r\R$ is the trivial representation of dimension $r$ (rank of $G$). 
Notice that the restriction of $\zeta_i$ to $T$ is isomorphic to the root space representation $\alpha_i$ (as a real representation). 

\medskip
\begin{defn} (The dual Adjoint sphere spectrum) \label{desusp} 

\noindent
Let $S^{-\zeta_i}$ denote the sphere spectrum for the virtual real $G_i$ representation $-\zeta_i$. In what follows, we may pick any model for this sphere. For instance, one may choose to define $S^{-\zeta_i}$ to be the dual of $\Sigma^{-r} S^{\mathfrak{g}_i}$, and denote the left $G_i$ action on it by $Ad(g)_\ast$ 
\[ S^{-\zeta_i} := \Map(S^{\mathfrak{g}_i}, S^r), \quad Ad(g)_\ast \, \varphi := g \circ \varphi \circ Ad(g^{-1}), \quad \varphi \in \Map(S^{\mathfrak{g}_i}, S^r). \]
\end{defn}

\medskip
\begin{defn} \label{BC} (Broken symmetries: Arbitrary braids)

\noindent
Consider the more general indexing sequence expressed as $I := \{ \epsilon_{i_1} i_1, \cdots, \epsilon_{i_k} i_k \}$, where $i_j \leq r_s$ as before, and $\epsilon_j = \pm 1$. Assume that $w = w_I := \sigma_{i_1}^{\epsilon_{i_1}} \cdots \sigma_{i_k}^{\epsilon_{i_k}}$. We define the equivariant $G$-spectrum  
\[  \BC(w_{I}) := G_+ \wedge_T \BC_T(w_I),  \, \, \mbox{where} \, \, \]
\[  \BC_T(w_{I}) := H_{i_1} \wedge_T \ldots  \wedge_T H_{i_k},  \, \, \mbox{and} \, \, H_i = S^{-\zeta_i} \wedge \, {G_i}_+,  \, \, \mbox{if} \, \, \epsilon_i = -1, \, \,  H_i = {G_i}_+ \, \,  \mbox{else}. \]
The $T \times T$-action on $H_i$ is defined by demanding that an element $(t_1, t_2) \in T \times T$ acts on $S^{-\zeta_i} \wedge \, {G_i}_+$ by smashing the action $Ad(t_1)_\ast$ on $S^{-\zeta_i}$ with the standard $T \times T$ action on ${G_i}_+$ given by left (resp. right) multiplication. As before the $T$-action on ${H_{i_1}}_+ \wedge_T {H_{i_2}}_+ \wedge_T \ldots  \wedge_T H_{i_k}$ is by conjugation on the first and last factor. Notice that $\BC(w_{I})$ is an equivariant Thom spectrum over the stack of broken symmetries $G \times_T(G_{i_1} \times_T \cdots \times_T G_{i_k})$. 
\end{defn}

\medskip
\noindent
Our eventual goal is to study the naturality properties of the construction $\BC(w_{I})$ in terms of subwords. In order to study this, we will require to make certain constructions known as Pontrjagin-Thom constructions that require studying tubular neighborhoods. In order to do so, let us fix a $G$-binvariant metric on $G$.  

\medskip
\begin{claim} \label{PT1}
The Pontrjagin-Thom construction along $T \subset G_i$ induces a canonical map of equivariant $T \times T$-spectra
\[ \pi_i : S^{-\zeta_i} \wedge \, {G_i}_+ \longrightarrow T_+, \]
where the $T \times T$-action on $T$ is induced by left (rep. right) group multiplication. 
\end{claim}
\begin{proof}
Let us describe the construction of the Pontrjagin-Thom construction along $T \subset G_i$ in some detail. First notice that the normal bundle of $T$ in $G_i$ is canonically trivial (using the right $T$-translation of the complement of the Lie algebra of $T$, which we have denoted by $\zeta_i$. The conjugation action of $T$ on this normal bundle can therefore be identified with the standard root-space action of $T$ on $\zeta_i$. Performing the Pontrjagin-Thom construction amounts to collapsing a complement of a epsilon neighborhood of $T$ (for some small epsilon fixed throughout). Doing so gives rise to the map 
\[ \pi_i : S^{-\zeta_i} \wedge \, {G_i}_+ \longrightarrow S^{-\zeta_i} \wedge S^{\zeta_i} \wedge T_+. \]
We would like to identify $S^{-\zeta_i} \wedge S^{\zeta_i}$ with $S^0$ so as to identify the codomain with $T_+$. This is clearly the case non-equivariantly but we must verify the required equivariance. Recall the action of $(t_1, t_2) \in T \times T$-action on $S^{-\zeta_i} \wedge \, {G_i}_+$ defined by smashing the action $Ad(t_1)_\ast$ on $S^{-\zeta_i}$ with the standard $T \times T$ action on ${G_i}_+$ given by left (resp. right) multiplication. Notice that $t_1 g_i t_2 = t_1 g_i t_1^{-1} t_1 t_2 = (Ad(t_1) g_i) \, t_1t_2$. In particular, performing the Pontrjagin-Thom construction on the inclusion $T \subset G_i$ turns the $T \times T$-action on ${G_i}_+$ into the expected action on $S^{\zeta_i} \wedge T_+$. It follows that the $T \times T$-action on the product $S^{-\zeta_i} \wedge S^{\zeta_i}$ is trivial, and amounts to group multiplication on the factor $T_+$ as we require. 
\end{proof}

\medskip
\noindent
\begin{defn} (The functor $\BC(w_J)$) \label{poset}

\noindent
Given a braid word $w_{I}$, for $I = \{ \epsilon_{i_1} i_1, \cdots, \epsilon_{i_k} i_k \}$, let $2^{I}$ denote the set of all the subsets of $I$. Let us define a poset structure on $2^{I}$ generated by demanding that nontrivial indecomposable morphisms $J \rightarrow K$ have the form where either $J$ is obtained from $K$ by dropping an entry $i_j \in K$ (i.e. an entry for which $\epsilon_{i_j} = 1$), or that $K$ is obtained from $J$ by dropping an entry $-i_j$ (i.e an entry for which $\epsilon_{i_j} = -1$). 

\medskip
\noindent
The construction $\BC(w_{J})$ induces a functor from the category $2^{I}$ to $G$-spectra. More precisely, given a nontrivial indecomposable morphism $J \rightarrow K$ obtained by dropping $-i_j$ from $J$, the induced map $\BC(w_{J}) \rightarrow \BC(w_{K})$ is obtained by applying the map $\pi_{i_j}$ of claim \ref{PT1} in the corresponding factor. Likewise, if $J$ is obtained from $K$ by dropping the factor $i_j$, then the map $\BC(w_{J}) \rightarrow \BC(w_{K})$ is defined as the canonical inclusion induced by the map $T_+ \rightarrow {G_{i_j}}_+$ in the corresponding factor. 
\end{defn}

\medskip
\noindent
For the functor $\BC(w_J)$ defined above, we will require the construction of a derived notion of a colimit over subcategories of $2^{I}$. This construction is known as the {\em homotopy colimit} (definition \ref{hocolim}), which we take as a black-box construction for now, and review the Appendix. We now define the equivariant $G$-spectrum of {\em{strict broken symmetries}}:

\bigskip
\begin{defn} \label{SBC} (Limiting strict broken symmetries) 

\noindent
let $I^+ \subseteq I$ denote the terminal object of $2^{I}$ given by dropping all terms $-i_j$ from $I$ (i.e. terms for which $\epsilon_{i_j} = -1$). Define the poset category $\I$ to the subcategory of $2^{I}$ given by removing the terminal object $I^+$. 
\[ \I = \{ J \in 2^{I}, \,  J \neq I^+ \} \]

\noindent
The equivariant $G$-spectrum $\sBCU(w_{I})$ is defined to be the cofiber of the canonical map: 
\[  \pi : \hocolim_{J \in \I} \BC(w_{J}) \longrightarrow \BC(w_{I^+}). \]
In other words, one has a cofiber sequence of equivariant $G$-spectra
\[ \hocolim_{J \in \I} \BC(w_{J}) \longrightarrow \BC(w_{I^+}) \longrightarrow \sBCU(w_{I}). \]
Note that the above cofiber sequence can be induced from a cofiber sequence of $T \times T$-spectra. 
\end{defn}

\medskip
\begin{defn} \label{Filt} (Filtration of strict broken symmetries via sub-posets $\I^t \subseteq \I$)

\noindent
We endow $\sBCU(w_{I})$ with a natural filtration as $G$-spectra defined as follows. The lowest filtration is defined as
\[ F_0 \, \sBC(w_{I}) = \BC(w_{I^+}), \quad \mbox{and} \quad F_k \, \sBC(w_{I}) = \ast, \quad \mbox{for} \quad k < 0. \]
Higher filtrations $F_t$ for $t>0$ are defined as the cone on the restriction of $\pi$ to the subcategory $\I^t \subseteq \I$ consisting of objects no more than $t$ nontrivial composable morphisms away from $I^+$. 

\smallskip
\noindent
In other words, the filtered spectrum of strict broken symmetries $F_t \, \sBC(w_{I})$ is defined via the cofiber sequence
\[ \hocolim_{J \in \I^t} \BC(w_{J}) \longrightarrow \BC(w_{I^+}) \longrightarrow F_t \, \sBC(w_{I}). \]
As before, note that $F_t \, \sBC(w_I) = G_+ \wedge_T F_t \, \sBC_T(w_I)$, with the filtered $T \times T$-spectrum $F_t \, \sBC_T(w_I)$ defined just as above. Since $\I^t \subset \I^{t+1}$, we obtain a canonical filtration of length $k = |I|$
\[  \BC(w_{I^+}) = F_0 \, \sBC(w_{I}) \rightarrow F_1 \, \sBC(w_{I}) \rightarrow F_2 \,  \sBC(w_{I}) \cdots \rightarrow F_k \,  \sBC(w_{I}) =  \sBCU(w_{I}). \]
\end{defn}

\medskip
\begin{remark}
The poset $\I$ is an iterated join of the trivial one-element poset (see remark \ref{join}). It follows that $\overline{\mathcal{I}}$ is a finite CW poset. It is now straightforward to see from theorem \ref{CWP} that the associated graded of the filtration described above is given by the cofiber sequence
\[ F_{t-1} \,  \sBC(w_{I}) \longrightarrow F_t \,  \sBC(w_{I}) \longrightarrow \bigvee_{{J} \in \I^t/\I^{t-1} }\Sigma^t \BC(w_{J}). \]
\end{remark}

\medskip
\noindent
For the purposes of the rest of the article, it is helpful to normalize definition \ref{SBC}. 

\medskip
\begin{defn} \label{stableLI} (Normalized strict broken symmetries)

\noindent
Given an element $w \in \Br(G)$, we define the normalized $G$-spectrum of strict broken symmetries
\[ \sNBC(w) := \Sigma^{l(w_I)} \sBC(w_I)[\varrho_I], \]
where $w_I$ is any presentation of $w$, and $l(w_I)$ stands for the integer $l_-(w_I) - 2l_+(w_I)$ with $l_+(w_I)$ denoting the number of positive exponents and $l_-(w_I)$ denoting the number of negative exponents in the presentation $w_I$ for $w$ in terms of the generators $\sigma_i$. Here $\sBC(w_I)[\varrho_I]$ denotes the shift in indexing given by $F_t \, \sBC(w_I)[\varrho_I] := F_{t+\varrho_I} \, \sBC(w_I)$ where the integer $\varrho_I$ that induces the shift is given by one-half the difference between the cardinality of the set $I$, $|I|$, and the mimimal word length $|w|$, of $w \in \Br(G)$
\[ \varrho_I = \frac{1}{2} (|I| - |w|). \]
\end{defn}

\medskip
\noindent
The notation for the normalization given above depends only on $w \in \Br(G)$, but makes no reference to the presentation of $w$. For this definition to make sense, one needs to verify that the filtered spectrum $\sNBC(w)$ is independent of the presentation. That is in fact the case, but the proof of that fact will take up the next several sections

\medskip
\begin{thm} \label{main1}
Given a braid element $w \in \Br(G)$, with two presentations $w = w_I = w_{I'}$, then the filtered $G$-spectra $\Sigma^{l(w_I)} \sBC(w_I)[\varrho_I]$ and $\Sigma^{l(w_{I'})} \sBC(w_{I'})[\varrho_{I'}]$ are quasi-equivalent, where quasi-equivalence is defined in definition \ref{quasi-equiv}.
\end{thm}

\medskip
\noindent
Before we move ahead to the next section that describes the notion of equivalence we work with, let us actually study the underlying homotopy type of the top filtration given by $\Sigma^{l(w_I)} \sBCU(w_I)$. In fact, this homotopy type has a nice description in terms of a Thom spectrum. We have

\medskip
\begin{thm} \label{main1a}
Given an indexing set $I = \{\epsilon_1 i_1, \ldots, \epsilon_k i_k \}$, Let $V_I$ denote the vitrual representation of $T$ given by a vitual sum of root spaces
\[ V_I = \sum_{j \leq k} \epsilon_j w_{I_{j-1}}(\alpha_{i_j}) , \quad \mbox{where} \quad w_{I_{j-1}} = \sigma_{i_1} \ldots \sigma_{i_{j-1}}, \quad w_{I_0} = id, \]
and where $w_{I_{j-1}} (\alpha_{i_j})$ denotes the root space for the root given by the $w_{I_{j-1}}$-translate of the simple root $\alpha_{i_j}$. Let $|V_I|$ denote the virtual dimension of $V_I$. Then the $G$-equivariant homotopy type of $\Sigma^{l(w_I)} \sBCU(w_I)$ is given by the equivariant Thom spectrum 
\[ \Sigma^{l(w_I)} \sBCU(w_I) = \Sigma^{-|V_I|} G_+ \wedge_T (S^{V_I} \wedge T(w)_+), \]
where $S^{V_I}$ denotes the sphere spectrum for the virtual $T$-representation $V_I$, and $T(w)$ denotes the twisted conjugation action of $T$ on itself
\[ t (\lambda) := w^{-1} t w t^{-1} \,  \lambda \quad \mbox{where} \quad w = \sigma_{i_1} \ldots \sigma_{i_k}, \quad t \in T, \quad \lambda \in T(w). \]
\end{thm}
\begin{proof}
Recall the definition of $\sBCU(w_I)$ via the cofiber sequence
\[ \hocolim_{J \in \I} \BC(w_{J}) \longrightarrow \BC(w_{I^+}) \longrightarrow \sBCU(w_{I}), \]
where $I^+$ is the terminal element of the set $2^I$ as defined in \ref{poset}, and $\I$ denotes the subcategory of $2^I$ consisting of all objects besides $I^+$. We may describe $\sBCU(w_{I})$ as a pushout 
\[ \BC(w_{I^+})\longleftarrow  \hocolim_{J \in \I} \BC(w_{J}) \longrightarrow  \ast. \]
Alternatively, $\sBCU(w_{I})$ can be seen as a homotopy colimit of the canonical functor that extends $\BC(w_J)$ over the poset $\Sigma \I$ given by the suspension of $\I$ obtained by adjoining two new vertices $\{ 0, \infty \}$ that admit morphisms from each object of $\I$. The poset $\Sigma \I$ can be described as a quotient of the poset $P^I$, where $P$ is the pushout category with three objects $\{ 0, 1, \infty \}$, and $P^I$ is the $k$-fold product of $P$. We think of objects of $P^I$ as functions on $I$ with values in the object set of $P$ and morphisms $\varphi < \psi$ if $\varphi(i) \leq \psi(i)$ for all $i \in I$. One recovers $\Sigma \I$ by identifying all functions taking the value $\infty$ on any element of $I$ with a distinguished object (also called $\infty$). An explicit identification of this quotient with $\Sigma \I$ is given by defining the function $\varphi \in P^I$ corresponding to the subset $J = \{ \epsilon_{i_{j_1}} i_{j_1}, \ldots \epsilon_{i_{j_s}} i_{j_s} \} \subseteq I$ as $\varphi(\epsilon_s i_s) = 1$ if $\epsilon_s i_s \in J$ and $\epsilon_s < 0$ or $\epsilon_s i_s \notin J$ and $\epsilon_s > 0$. Define $\varphi(\epsilon_s i_s) = 0$ otherwise. In particular, the maximal subset $I^+$ represents the unique function taking the value $0$ on all elements of $I$. Now recall that the spectrum $\BC(w_J)$ 
\[  \BC(w_{J}) := G_+ \wedge_T \BC_T(w_J),  \, \, \mbox{where} \, \, \]
\[  \BC_T(w_{J}) := H_{i_{j_1}} \wedge_T \ldots  \wedge_T H_{i_{j_s}},  \, \, \mbox{and} \, \, H_i = S^{-\zeta_i} \wedge \, {G_i}_+,  \, \, \mbox{if} \, \, \epsilon_i = -1, \, \,  H_i = {G_i}_+ \, \,  \mbox{else}. \]
Since $P^I$ is a product of pushout categories, and the functor describing $\sBCU(w_I)$ decomposes as a smash product of functors, we may express the homotopy colimit as
\[ \sBCU(w_I) = G_+ \wedge_T P(H_{i_1}) \wedge_T \ldots  \wedge_T P(H_{i_k}), \]
where $P(H_i)$ are the pushout diagrams constructed from the elementary morphisms described in definition \ref{poset}. More precisely, if $\epsilon_i = -1$, then we have
\[
\xymatrix{
 S^{-\zeta_i} \wedge \, {G_i}_+ \ar[d] \ar[r] &  T_+ \ar[d] \\
\ast     \ar[r] & P(H_i).
}
\]
Since $S^{\zeta_i} \wedge T_+$ is obtained from $G_i$ by pinching out the $T \times T$-subspace given by the left coset $\sigma_i T$ (see proof of claim \ref{PT1}), the above pushout is equivalent to the spectrum $\Sigma S^{-\zeta_i} \wedge \sigma_i \, T_+$. On the other hand, if $\epsilon_i = 1$, then the pushout diagram is given by
\[
\xymatrix{
 T_+ \ar[d] \ar[r] &  {G_i}_+ \ar[d] \\
\ast     \ar[r] & P(H_i).
}
\]
which is equalent to the spectrum $S^{\zeta_i} \wedge \sigma_i T_+$, again using the proof of claim \ref{PT1}. Since the $T$-representation $\zeta_i$ is isomorphic to $\alpha_i$, we may smash them together to obtain
\[ \sBCU(w_I) = \Sigma^m G_+ \wedge_T (S^{\epsilon_1\zeta_{i_1}} \wedge \sigma_{i_1} T_+) \wedge_T \ldots  \wedge_T (S^{\epsilon_k\zeta_{i_k}} \wedge \sigma_{i_k} T_+), \]
where $m$ is the number of negative exponents. Now suspending by a sphere of dimension $l(w_I)$, and collecting all the terms $\sigma_i T_+$ is easily seen to yield the result we seek to prove. 
\end{proof}

\section{Some filtered algebra} \label{Filteredalgebra}

\medskip
\noindent
Before we address the the particular properties enjoyed by $\sBC(w_I)$, let us digress briefly into the theory of filtered $G$-spectra so as to define a lax notion of equivalence of filtered equivariant $G$-spectra that would be relevant for our purposes. 

\medskip
\noindent
By a bounded-below filtered $G$-spectrum $X := \{ F_t X \}$, we mean a filtered $G$-spectrum with the filtration being the trivial spectrum below some fixed integer $n$. We have a collection of cofiber sequences:
\[  \cdots F_t X \rightarrow F_{t+1} X \rightarrow F_{t+1} X / F_t X \rightarrow \Sigma F_t X \cdots \]
which assemble to a collection of maps 
\[ \partial_t : F_{t+1} X/ F_t X  \longrightarrow \Sigma F_t X \longrightarrow \Sigma (F_t X / F_{t-1} X). \]
Furthermore, $\partial_{t-1} \circ \partial_t$ is null homotopic for each $i$. In particular one obtains a bounded-below graded chain complex in the homotopy category of $G$-spectra associated to $X$.

\medskip
\begin{defn} (The associated graded and the shift functor for filtered spectra) \label{shift}

\noindent
The associated graded chain complex of a bounded-below filtered $G$-spectrum $X$ is defined as
\[ \{ \Gr_t (X) \} := \{ \Sigma^{-t} (F_t X/ F_{t-1} X), \partial_{t-1} \}. \]
Given a bounded below filtered $G$-spectrum $X$, we define the shifted spectrum $X[\varrho]$ as
\[ F_t X[\varrho] := F_{t+\varrho} X. \]
\end{defn}

\begin{remark} \label{regrading}
Notice that desuspension and shift together amount to a reindexing of the associated graded complex. In other words, we have $\{ \Gr_t (\Sigma^{-\varrho} X[\varrho]) \} = \{ \Gr_{t+\varrho} (X) \}$. 
\end{remark}

\medskip
\begin{example} \label{FiltBC}
The associated graded spectrum associated to $\sBC(w_I)$ is given by
\[ \Gr_t (\sBC(w_I)) = \bigvee_{J \in \I^t/\I^{t-1}} \BC(w_{J}), \]
with $\partial_t$ being induced by a signed sum along nontrivial indecomposable morphisms. 
\end{example}

\medskip
\begin{defn} \label{quasi-equiv} (Acyclicity and quasi-equivalence of filtered $G$-spectra) 

\noindent
A filtered $G$-spectrum $X$ is said to be acyclic if the associated graded complex $\Gr_t(X)$ admits a ``null chain homotopy" $h_t$ for all $t \geq 0$, i.e. one whose graded commutator with $\partial$ is an equivalence
\[ h_t : \Gr_t (X)  \longrightarrow \Gr_{t+1} (X), \quad \partial_t \circ h_t + h_{t-1} \circ \partial_{t-1} : \Gr_t (X) \llra{\simeq} \Gr_t (X). \]
A map of filtered $G$-spectra $\rho : X \longrightarrow Y$ is defined as a collection of maps $\rho_t : F_t X \longrightarrow F_t Y$, compatible with the filtration. The  map $\rho$ is said to be an elementary quasi-equivalence if the filtered spectrum defined by the fibers (or cofibers) of $\rho_t$, is acyclic. Two filtered $G$-spectra are said to be quasi-equivalent if they are connected by a zig-zag of elementary quasi-equivalences. We will refer to usual levelwise equivalences as honest equivalences. 
\end{defn}

\medskip 
\begin{remark} \label{limiting}
Notice that the definitions imply that if two spectra $X$ and $Y$ with finite filtrations are quasi-equivalent, then their limiting $G$-spectra $X_\infty := \hocolim_t F_t X$ and $Y_\infty := \hocolim_t F_t Y$ respectively are $G$-equivariantly homotopy equivalent. The converse need not be true.
\end{remark}

\medskip
\begin{claim} \label{q-isom}
Let $\rho : X \longrightarrow Y$ be an elementary quasi-equivalence of filtered $G$-spectra. Given a $G$-equivariant cohomology theory $\E_G$ so that the map of cochain complexes induced by $\rho$ 
\[ \rho^* : \E^*_G(\Gr_t(Y)) \longrightarrow \E^*_G(\Gr_t (X)), \]
is either injective for all $t$, or surjective for all $t$. Then $\rho^*$ is a quasi-isomorphism. 
\end{claim}
\begin{proof}
The proof is straightforward. We prove it under the injectivity assumption, the surjective case is analogous. Assuming injectivity, we notice that there is a short exact sequence of cochain complexes
\[ 0 \rightarrow \E^*_G(\Gr_t(Y)) \llra{\rho^\ast} \E^*_G(\Gr_t(X)) \longrightarrow \E^*_G(\Gr_t(Z)) \rightarrow 0,  \]
where $Z := \{ F_t Z \}$ is the fiber of $\rho$. By definition of acyclicity, the cochain complex $\E^*_G(\Gr_t(Z))$ is acyclic. The long exact sequence in cohomology therefore implies that $\rho^*$ is a quasi-isomorphism. 
\end{proof}

\medskip
\noindent
It is easy to see why the requirement of injectivity or surjectivity in the above claim is necessary. For example, given a filtered spectrum $X$, consider the canonical map of filtered spectra given by the shift that maps each filtrate into the next
\[ s : X \longrightarrow X[1], \quad \quad F_t X \longrightarrow F_{t+1} X. \]
It is straightforward to see that $s$ is an elementary quasi-equivalence. However, the induced map $s^\ast$ is trivial on the associated graded complex in any cohomology theory $\E_G$. We show in the following claim that the above example is universal in a suitable sense.

\medskip
\begin{claim} \label{q-isom2}
Let $\rho : X \longrightarrow Y$ be an elementary quasi-equivalence of filtered $G$-spectra, then there exists a filtered $G$-spectrum $P_\rho$ endowed with elementary quasi-equivalences
\[ \iota_Y : Y \longrightarrow P_\rho, \quad \quad \iota_X : X[1] \longrightarrow P_\rho, \quad \quad \mbox{with} \quad \quad \iota_X \circ s \simeq \iota_Y \circ \rho, \]
where $s$ is the shift map defined above. In particular, $P_\rho$ furnishes a quasi-equivalence between $Y$ and $X[1]$. Furthermore, given a $G$-equivariant cohomology theory $\E_G$, if the map of the associated graded cochain complexes induced by $\rho$
\[ \rho^* : \E^*_G(\Gr_t(Y)) \longrightarrow \E^*_G(\Gr_t (X)), \]
 is trivial, then both maps $\iota_Y^\ast$ and $\iota_X^\ast$ are quasi-isomorphisms on the associated graded complex. 
\end{claim}
\begin{proof}
Let us define the filtered $G$-spectrum by defining $\{ F_t P_\rho \}$ as the homotopy pushout

\[
\xymatrix{
 F_t X \ar[d]^{\rho} \ar[r]^s & F_t X[1] \ar[d]^{\iota_X} \\
 F_t Y   \ar[r]^{\iota_Y} & F_t P_\rho.
}
\]
By construction, the fiber of $\iota_X$ is the fiber of $\rho$, which is acyclic. Similarly, the fiber of $\iota_Y$ is the fiber of $s$, which is also acyclic. Hence, both $\iota_X$ and $\iota_Y$ are elementary quasi-equivalences. Now assume that $\rho$ is trivial in $\E_G$-cohomology. Then it is easy to see from comparing the long-exact sequences in cohomology for the two rows (resp. columns), that $\iota_X^\ast$ (resp. $\iota_Y^\ast$) is surjective on the associated graded complex. By claim \ref{q-isom}, it follows that they are quasi-isomorphisms. 
\end{proof}

\medskip
\section{Properties of strict broken symmetries: The Markov 1 property} \label{Markov 1}

\medskip
\noindent
Beginning with this section we shall prove various helpful properties of the $G$-spectrum of strict broken symmetries. To begin with, we will prove a Markov 1 type result which essentially says that $\sBC(w_{I})$ is equivalent to the spectrum $\sBC(w_{I^{\tau}})$, where $I^{\tau}$ is the sequence obtained by cyclicly permuting $I$. More precisely
\begin{defn} (The permuted poset $I^{\tau}$)

\noindent
Given an indexing sequence $I = \{ \epsilon_1 i_1, \epsilon_2 i_2, \ldots, \epsilon_k i_k \}$, we define the sequence
\[ I^{\tau} =  \{ \epsilon^\tau_1 i^\tau_1, \epsilon^\tau_2 i^\tau_2, \ldots, \epsilon^\tau_k i^\tau_k\} := \{ \epsilon_2 i_2, \epsilon_3 i_3, \ldots, \epsilon_k i_k, \epsilon_1 i_1 \}. \]
Given a subset $J \in 2^{I}$, let the image of $J$ under this permutation be denoted by $J^\tau \in 2^{I^{\tau}}$. This gives rise to an isomorphism of the posets $\tau : 2^{I} \longrightarrow 2^{I^{\tau}}$ which restricts to an isomorphism $\tau : \I \longrightarrow \I^\tau$, where $\I^\tau$ is the poset of all non-terminal objects in $2^{I^{\tau}}$. 
\end{defn}

\medskip
\noindent
With the above definition in place, we prove the Markov 1 property: 

\medskip
\begin{thm}\label{M1} The functors $\BC(w_{J})$ and $\BC(w_{J^\tau}) \circ \tau$ are equivalent. In particular, $\tau$ induces a levelwise (honest) equivalence of filtered $G$-spectra
\[ \tau_t : F_t \, \sBC(w_{I}) \llra{\simeq} F_t \, \sBC(w_{I^\tau}), \quad t \geq 0. \]
\end{thm}
\begin{proof}
We require a natural equivalence between the $G$-spectra $\BC(w_{J})$ and $\BC(w_{J^\tau})$. Let $J = \{\epsilon_{i_{j_1}} i_{j_1}, \ldots, \epsilon_{i_{j_s}} i_{j_s} \}$ be an element in $2^{I}$, so that $\tau(J) : = J^\tau$ is defined as follows
\[ J^\tau = \{ \epsilon^\tau_{i_{j_1-1}} i^\tau_{j_1-1}, \epsilon^\tau_{i_{j_2-1}} i^\tau_{j_2-1}, \ldots, \epsilon^\tau_{i_{j_s-1}} i^\tau_{j_s-1} \} \subseteq I^\tau, \quad \mbox{if} \quad j_1 > 1, \quad \mbox{and} \]
 \[ J^\tau = \{ \epsilon^\tau_{i_{j_2-1}} i^\tau_{j_2-1}, \ldots, \epsilon^\tau_{i_{j_s-1}} i^\tau_{j_s-1}, \epsilon^\tau_{k} i^\tau_k \} \subseteq I^\tau, \quad \mbox{if} \quad j_1=1.\]
Recall from definition \ref{BC} that 
\[  \BC(w_{J}) : G_+ \wedge_T (H_{i_{j_1}} \wedge_T H_{i_{j_2}} \wedge_T \ldots  \wedge_T H_{i_{j_s}}),  \, \, \mbox{where} \, \, H_i = S^{-\zeta_i} \wedge \, {G_i}_+,  \, \, \mbox{if} \, \, \epsilon_i = -1, \, \,  H_i = {G_i}_+ \, \,  \mbox{else}. \]
Let us first consider the case of a positive braid $w_{I}$, so that $H_{i_j} = {G_{i_j}}_+$ for all $j$. In that case, we define $\tau$ on the underlying topological space by
\[ \tau [(g, g_{i_{j_1}}, \ldots, g_{i_{j_s}})] = [(g , g_{i_{j_1}}, \ldots, g_{i_{j_s}})], \quad \mbox{if} \quad j_1 > 1, \quad \mbox{and} \]
\[ \tau [(g, g_{i_{j_1}}, \ldots, g_{i_{j_s}})] = [(gg_{i_1} , g_{i_{j_2}}, \ldots, g_{i_{j_s}}, g_{i_1})], \quad \mbox{if} \quad j_1=1,\]
where the right hand side is indexed by the subset $J^\tau := \tau(J)$. 

\medskip
\noindent
It is easy to see that this map is well defined. The map defined above extends (by permuting the equivariant spheres) to the equivariant vector bundle obtained by smashing this space with the equivivariant spheres of the form $S^{-\zeta_i}$. In other words, one may replace ${G_{i_j}}$ by $H_{i_j}$ in the above description to obtain a map covering the space level map described above. This defines the natural equivalence of functors we seek
\[ \tau : \BC(w_{J}) \longrightarrow \BC(w_{J^\tau}). \]
\end{proof}

\section{Properties of strict broken symmetries: Braid invariance} \label{Braid}

\medskip
\noindent
We now move towards showing that $\sBC(w_{I})$ depends only on the braid element $w$ and not on the presentation $I$ used to express it. This property requires proving two results. The first result would require showing that $\sBC(w_{I})$ is invariant under the braid relations, and the second result would require us to show invariance under the inverse relation, namely that one may contract successive terms in $I$ of the form $\{ \ldots, -i, i, \ldots \}$ or $\{ \ldots, i, -i, \ldots \}$ without changing the equivariant homotopy type up to quasi-equivalence. 

\medskip
\noindent
We begin our goal by first proving the following theorem on braid invariance:

\medskip
\begin{thm} \label{braid rel} For a fixed pair of indices $(i,j)$, let $I^{(i,j)}$ be an arbitrary sequence that contains the braid sequence $\mathcal{O}^{(i,j)} := \{ i,j,i,j, \ldots \}$ with $m_{ij}$-terms as a subsequence of consecutive terms. Define $I^{(j,i)}$ to be the sequence obtained by replacing the braid subsequence with its counterpart $ \{ j,i,j,i \ldots \}$ with $m_{ij}$-terms. Then the filtered $G$-spectra $\sBC(w_{I^{(i,j)}})$ and $\sBC(w_{I^{(j,i)}})$ are connected by a sequence of zig-zags of elementary quasi-equivalences. In particular, they are quasi-equivalent. 
\end{thm}

\medskip
\noindent
We will only consider the nontrivial case where $m_{i,j} > 2$. The case when $m_{i,j} = 2$ is straightforward since the (left/right) $T \times T$-spaces $G_i \times_T G_j$ and $G_j \times_T G_i$ can both be identified with the same space, namely the subgroup of $G$ generated by $G_i$ and $G_j$. The proof of theorem \ref{braid rel} for $m_{i,j} > 2$ is fairly technical, and packed with several constructions and corresponding definitions. The reader interested in the bigger picture may safely ignore the rest of the section. For those who choose the path of most resistance, it would be helpful to review the Appendix (section \ref{appendix}) briefly before proceeding. The general argument of the proof can be outlined as follows. 

\medskip
\noindent
We will introduce two sequences of filtered $G$-spectra $\sBS^{(i,j,m)}(w_I)$ and $\sBS^{(j,i,m)}(w_I)$ resp. for $1 \leq m \leq m_{ij}$ called {\em strict broken Schubert} spectra. The spectra in either sequence will be shown to belong to the same quasi-equivalence class by a sequence of zig-zags of elementary quasi-equivalences. Moreover, by construction, the sequences will begin with $\sBC(w_{I^{(i,j)}})$ and $\sBC(w_{I^{(j,i)}})$ respectively, and end with the exact same spectrum, allowing us to deduce the quasi-equilance between $\sBC(w_{I^{(i,j)}})$ and $\sBC(w_{I^{(j,i)}})$. In other words, the beginning and ending terms of the sequences are

\medskip
\[ \sBS^{(i,j,1)}(w_I) = \sBC(w_{I^{(i,j)}}), \quad \sBS^{(j,i,1)}(w_I) = \sBC(w_{I^{(j,i)}}), \]
\[ \sBS^{(i,j,m_{ij})}(w_I) = \sBS^{(j,i,m_{ij})}(w_I). \]

\medskip
\noindent
As with strict broken symmetries, the filtered spectra $\sBS^{(i,j,m)}(w_I)$ and $\sBS^{(j,i,m)}(w_I)$ will be defined via a homotopy colimit of two functors $\BS^{(i,j,m)}(w_J)$ and $\BS^{(j,i,m)}(w_J)$ (resp.) indexed over a poset category. Before we do that, we require

\medskip
\begin{defn} (Schubert varieties and their lifts) \label{Schubert}

\noindent
Let $\mathcal{O}^{(i,j)}$ denote the indexing sequence $\{ i, j, i, \ldots \}$ ($m_{ij}$-terms). Let $K \subseteq \mathcal{O}^{(i,j)}$ be any subset $K = \{ k_1, k_2, \ldots, k_q \}$. The Schubert variety $\mathcal{X}_K$ is defined as the image under group multiplication
\[ \mathcal{X}_K = \mbox{Image of} \, \, \, \,  G_{k_1} \times_T \cdots \times_T (G_{k_s}/T) \longrightarrow G/T. \]
Define $\Sh_K \subset G$ to be the $T \times T$-invariant subspace to be the pullback (where $T \times T$ acts on $G$ via left/right multiplication)
\[
\xymatrix{
\Sh_K   \ar[d] \ar[r] &  G \ar[d] \\
\mathcal{X}_K     \ar[r] & G/T.
}
\]
\end{defn}

\begin{remark} \label{reduced}
Notice that $\Sh_K$ depends only on the {\bf reduced sequence for $K$}, namely the sequence obtained by contracting all sequentially repeated elements. For instance $\Sh_K$ for the subsequence $K = \{ i,j,j,i \} \subset \mathcal{O}^{(i,j)} := \{i,j,i,j,i\}$ agrees with $\Sh_{K_{red}}$, where $K_{red} = \{i,j,i\}$.
\end{remark}

\medskip
\begin{defn} \label{PORS-2} (Poset of reduced sequences)

\noindent
Let $\mathcal{O}^{(i,j)}$ denote the indexing sequence $\{ i, j, i, \ldots \}$ ($m_{ij}$-terms). Let $\mathcal{O}_{red}^{(i,j)}$ denote the quotient of the poset of all subsets of ${\mathcal{O}^{(i,j)}}$ under the equivalence relation that identifies two indexing sequences if they have the same reduced sequence (see remark \ref{reduced} above).  

\medskip
\noindent
For any $0 < p < m_{ij}$, we have exactly two elements of $\mathcal{O}_{red}^{(i,j)}$ given by reduced sequences of length $p$, namely $\{ ijij\ldots \}$ or $\{ jiji\ldots \}$. There are two more additional sequences given by the empty sequence and the sequence $\{ ijij \ldots \}$ with $m_{ij}$-terms. There is a unique nontrivial morphism from one sequence into a strictly longer sequence. 

\medskip
\noindent
For $1 \leq m \leq m_{ij}$, let $\mathcal{O}^{(i,j,m)} \subseteq \mathcal{O}^{(i,j)}$ denote the indexing sequence containing the last $m$-terms, and let $\mathcal{O}_{red}^{(i,j,m)} \subseteq \mathcal{O}_{red}^{(i,j)}$ denote the sub-poset of sequences in the equivalence class of those subsets in $\mathcal{O}^{(i,j,m)}$. In particular, $\mathcal{O}_{red}^{(i,j,m)}$ has $2m$ elements. Recall the poset $\mathcal{P}_{(m-1)}$ that captures the standard regular CW decomposition of the $(m-1)$-disc (see example \ref{subex}). It is easy to see that $\mathcal{O}_{red}^{(i,j,m)}$ is equivalent to the poset $\mathcal{P}_{(m-1)} \cup \infty $ obtained by adding a terminal object to $\mathcal{P}_{(m-1)}$. 
\end{defn}

\noindent
Let us briefly explore sequences of categories that contain the above posets. 

\medskip
\begin{defn} \label{PORS} (Indexing sequences containing posets of reduces sequences)

\noindent
Consider an indexing sequence $I^{(i,j)}$ that contains the subsequence $\mathcal{O}^{(i,j)}$, so that we have
\[ I^{(i,j)} = \{ \epsilon_1 i_1, \ldots, \epsilon_l i_l, \mathcal{O}^{(i,j)}, \epsilon_{l+m_{ij}+1} i_{l+m_{ij}+1}, \ldots, \epsilon_k i_k \} = I^{(i,j,m_{ij})} \coprod \mathcal{O}^{(i,j)}, \]
where we define $I^{(i,j,m_{ij})} = \{ \epsilon_1 i_1, \ldots, \epsilon_l i_l,  \epsilon_{l+m_{ij}+1} i_{l+m_{ij}+1}, \ldots, \epsilon_k i_k \}$. Similarly, we define $I^{(i,j,m)}$ by the presentation $I^{(i,j,m)} = \{ \epsilon_1 i_1, \ldots, \epsilon_{l+m_{ij}-m} i_{l+m_{ij}-m}, \epsilon_{l+m_{ij}+1} i_{l+m_{ij}+1}, \ldots, \epsilon_k i_k \}$, so that $ I^{(i,j)} = I^{(i,j,m)} \coprod \mathcal{O}^{(i,j,m)}$. 

\smallskip
\noindent
Define the poset categories $\Jj^{(i,j)}_m = 2^{I^{(i,j,m)}} \times \mathcal{O}^{(i,j,m)}_{red}$, and $\I_m = \Jj^{(i,j)}_m/J^{(m,+)}$, where $J^{(m,+)}$ is the terminal object. It follows from defintion \ref{PORS-2} and \ref{join} that the poset $\I_m$ is an iterated join of $\mathcal{P}_{(m-1)}$ with several factors of the one element poset $\bullet$. Notice that for $m < m_{ij}$ there is a canonical projection map of posets $\pi_m : \I_m \longrightarrow \I_{m+1}$. Moreover, it satisfies the property that the preimage of any element is unique unless it is of the form $(J,K) \in 2^{I^{(i,j,m+1)}} \times \mathcal{O}^{(i,j,m+1)}_{red}$ for which the sequence $K$ begins with $i_{l+m_{ij}-m}$, and is non maximal (so that it can be augmented on the left). It this case, the preimage of $(J,K)$ consists of three elements: $(J,K)$, $(\tilde{J}, \tilde{K})$, and $(\tilde{J}, K)$ where $\tilde{J}$ augments $J$ by the element $i_{l+m_{ij}-m}$, and $\tilde{K}$ is obtained from $K$ by dropping the initial term. One checks that the third element is the unique element that lies above other two by virtue of an indecomposable inclusion. An easy exercise now shows that $\pi_m$ is a subdivision of posets as defined in \ref{refine} (also see example \ref{subex}).  
\end{defn}

\noindent
We finally get to the definition of the family of filtered $G$-spectra $\sBS^{(i,j,m)}(w_I)$. 

\medskip
\begin{defn} (Strict broken Schubert spectra)

\noindent
Define functors of broken Schubert spectra $\BS^{(i,j,m)}(w_{(J,K)})$ on the category $\Jj^{(i,j)}_m$ as follows: 
\[ \BS^{(i,j,m)}(w_{(J,K)}) = G_+ \wedge_T (H_{i_{j_1}} \wedge_T \cdots \wedge_T H_{i_{j_q}} \wedge_T {\Sh_{K}}_+ \wedge_T H_{i_{j_{q+1}}} \wedge_T \cdots \wedge_T H_{i_{j_s}}), \]
where $J \in 2^{I^{(i,j,m)}}$ and $K \in \mathcal{O}_{red}^{(i,j,m)}$. We also recall our convention that $H_i = S^{-\zeta_i} \wedge \, {G_i}_+$ if $\epsilon_i = -1$, and $H_i = {G_i}_+$ if $\epsilon_i = 1$. 
As in the case of strict broken symmetries, we define the strict broken Schubert spectra as the filtered equivariant $G$-spectra $F_t \, \sBS^{(i,j,m)}(w_{I})$ by the cofiber sequence
\[ \hocolim_{(J,K) \in \I_m^t} \BS^{(i,j,m)}(w_{(J,K)}) \longrightarrow \BS^{(i,j,m)}(w_{J^{(m,+)}}) \longrightarrow F_t \, \sBS^{(i,j,m)}(w_{I}), \]
where $\I_m^t$ we recall is the sub poset of elements in $\I_m$ that are no more than $t$ non-trivial decomposable morphisms away from the terminal object $J^{(m,+)}$. Notice that the filtered spectrum $\sBS^{(i,j,1)}(w_I)$  agrees with $\sBC(w_{I^{(i,j)}})$, and that we have $\sBS^{(i,j,m_{ij})}(w_{I}) = \sBS^{(j,i,m_{ij})}(w_{I})$ since the functor $\BS^{(i,j,m_{ij})}(w_{(J,K)})$ is the same in both cases. 
\end{defn}

\medskip
\noindent
As mentioned previously, we will presently show that all the filtered spectra of the type $ \sBS^{(i,j,m)}(w_{I})$ are in the same quasi-equivalence class. Of course, the same will be true for $(i,j)$ replaced by $(j,i)$. That would constitute the proof of theorem \ref{braid rel} as indicated above. However, we need a preliminary lemma that will help us compare pointwise fibers along a map of homotopy colimits. 

\medskip
\begin{lemma} \label{elem q-equiv}
Assume $K \subseteq \mathcal{O}^{(i,j)}$ is a subsequence so that $K = \{ k_1, k_2, \ldots, k_q \}$, with $k_{m+1} = \overline{k}_m$, where $\overline{k}_m$ is the counterpart of $k_m$. So for instance if $k_m = i$, then $\overline{k}_m = j$ and vice versa. Assume $X$ and $Y$ are $T \times T$-spaces so that $X$ is free as a right $T$-space. Then the following diagram is an honest pushout of equivariant $T \times T$-spaces
\[
\xymatrix{
X \times_T G_{k_1} \times_T \Sh_K \times_T Y  \ar[d] \ar[r] &  X \times_T G_{k_1} \times_T \Sh_{K'} \times_T Y \ar[d] \\
X \times_T \Sh_K \times_T Y    \ar[r] & X \times_T \Sh_{K''} \times_T Y, 
}
\]
where $K'$ is defined as the set $\{ \overline{k}_1, k_1, k_2, \ldots, k_q \}$, and $K'' = \{ k_1, \overline{k}_1, k_1, k_2, \ldots, k_q \}$. All maps in the above diagram are given by the canonical maps. The vertical maps being induced by multiplication in $G$, and the horizontal ones being the standard inclusion induced by $K \subset K' \subset K''$. 
\end{lemma}
\begin{proof}
Using the left freeness of $X$ as a right $T$-space, we see that the above diagram fibers over the following diagram, with fiber $Y$:
\[
\xymatrix{
X \times_T G_{k_1} \times_T \mathcal{X}_K \ar[d] \ar[r] &  X \times_T G_{k_1} \times_T \mathcal{X}_{K'}  \ar[d] \\
X \times_T  \mathcal{X}_K     \ar[r] & X \times_T \mathcal{X}_{K''}. 
}
\]
Again, using the freeness of $X$ as a right $T$-space, we see that the above diagram itself fibers over $X/T$, with fiber being the diagram:
\[
\xymatrix{
G_{k_1} \times_T \mathcal{X}_K \ar[d] \ar[r] &  G_{k_1} \times_T \mathcal{X}_{K'}  \ar[d] \\
\mathcal{X}_K     \ar[r] & \mathcal{X}_{K''}. 
}
\]
It is therefore enough to prove that the diagram shown above is an honest pushout of left $T$-spaces. This is essentially an application of the Bruhat decomposition theorem \cite{Ku}. The Bruhat decomposition theorem says that there is a canonical (left) $T$-equivariant CW decomposition of the Schubert varieties of the form $\mathcal{X}_K$. Furthermore, the (open) cells are a product (under group multiplication) of the 2-cells of the form $\C_k$, where $\C_k$ denotes any lift of the space $(G_k/T)-(T/T)$ to $G_k$. 

\medskip
\noindent
Let us make the Bruhat decomposition precise in the case of interest to us. Let $K_{red}$ denote the reduced set corresponding to $K$. Recall from remark \ref{reduced}, that $K_{red}$ is obtained from $K$ by contracting all repeated indices. Let $n$ be the cardinality of $K_{red}$. Bruhat decomposition then gives us a cellular decomposition of $\mathcal{X}_K$ with the top cell given by the alternating product $\C_{k_1} \times \C_{\overline{k}_1} \times \C_{k_1} \times \cdots$ $n$-terms. Lower dimensional open cells are given by alternating products of the form $\C_i \times \C_j \times \C_i \times \cdots$ or $\C_j \times \C_i \times \C_j \times \cdots $ with $p$ terms for $0 \leq p < n$. In particular, $K'$ is obtained from $K$ by adding two more cells given by $\C_{\overline{k}_1} \times \C_{k_1} \times \C_{\overline{k}_1} \times \cdots$ $n$-terms, and the cell  $\C_{\overline{k}_1} \times \C_{k_1} \times \C_{\overline{k}_1} \times \cdots$ $(n+1)$-terms. The space $G_{k_1} \times_T \mathcal{X}_{K'}$ is therefore obtained from $G_{k_1} \times_T \mathcal{X}_K$ by adding yet another two cells $\C_{k_1} \times \C_{\overline{k}_1} \times \C_{k_1} \times \C_{\overline{k}_1} \times \cdots$ $(n+1)$-terms, and the cell  $\C_{k_1} \times \C_{\overline{k}_1} \times \C_{k_1} \times \C_{\overline{k}_1} \times \cdots$ $(n+2)$-terms. It follows that the cofiber of the inclusion of $G_{k_1} \times_T \mathcal{X}_K \subset G_{k_1} \times_T \mathcal{X}_{K'}$ has a $T$-invariant CW decomposition with four cells given by $\C_{\overline{k}_1} \times \C_{k_1} \times \C_{\overline{k}_1} \times \cdots$ $n$-terms, the cell  $\C_{\overline{k}_1} \times \C_{k_1} \times \C_{\overline{k}_1} \times \cdots$ $(n+1)$-terms, the cell $\C_{k_1} \times \C_{\overline{k}_1} \times \C_{k_1} \times \C_{\overline{k}_1} \times \cdots$ $(n+1)$-terms, and the cell  $\C_{k_1} \times \C_{\overline{k}_1} \times \C_{k_1} \times \C_{\overline{k}_1} \times \cdots$ $(n+2)$-terms. These are precisely the four cells that build $\mathcal{X}_{K''}$ from $\mathcal{X}_K$. In particular, the cofibers of the horizontal maps in the above diagram are mapped homeomorphically under the vertical maps. This is equivalent to saying that the diagram is a pushout. 
\end{proof}

\medskip
\noindent
\begin{remark} \label{pushout}
Notice that the left vertical map in the diagram for lemma \ref{elem q-equiv} splits at $T \times T$-spaces, hence we have the equality in the category of $G$-spectra
\[ (X \times_T G_{k_1} \times_T \Sh_K \times_T Y)_+ = F \vee (X \times_T \Sh_K \times_T Y)_+ , \]
where $F$ is the equivariant $G$-spectrum given by the fiber of (either) vertical map. It is easy to verify that the statement of lemma \ref{elem q-equiv}, and the above splitting also holds if we replace each corner of the commutative diagram of \ref{elem q-equiv} by the Thom spectrum of a bundle $\zeta$ pulled back from the pushout $X \times_T \Sh_{K''} \times_T Y$. 
\end{remark}

\medskip
\noindent
Our next step is to show that $\sBS^{(i,j,m)}(w_I)$ and $\sBS^{(i,j,m+1)}(w_I)$ are quasi-equivalent. We do that by means of a zig-zag of elementary quasi-equivalences induced by the map of posets $\pi_m$ of definition \ref{PORS}. We abuse notation by overusing the notation $\pi_m$ for induced functors, hoping that the context avoids any confusion. For $m < m_{i,j}$, consider maps 
\[ \pi_m : \sBS^{(i,j,m)}(w_I) \longrightarrow \pi_m^* \sBS^{(i,j,m+1)}(w_I) \longleftarrow \sBS^{(i,j,m+1)}(w_I) : \iota_m, \]
where $\pi_m : \sBS^{(i,j,m)}(w_I) \longrightarrow \pi_m^* \sBS^{(i,j,m+1)}(w_I)$ is induced by the natural transformation between the functor $\BS^{(i,j,m)}(w_{(J,K)})$ and the functor $\pi_m^*\BS^{(i,j,m+1)}(w_{(J,K)}) = \BS^{(i,j,m+1)}(w_{\pi_m(J,K)})$. The map $\iota_m : \sBS^{(i,j,m+1)}(w_I)  \longrightarrow \pi_m^* \sBS^{(i,j,m+1)}(w_I)$ is constructed using the fact that $\pi_m$ is a subdivision as a map of posets (see definition \ref{PORS}), and then applying theorem \ref{subdiv} from the Appendix. 

\medskip
\noindent
We begin by analyzing the map $\pi_m$. Let $Z_m$ denote the filtered $G$-spectrum representing the fiber of $\pi_m$. Consider the fibration induced by $\pi_m$ on the level of associated graded
\[ \Gr_t (Z_m) \longrightarrow \bigvee_{(J,K) \in \I_m^t/\I_m^{t-1}} \BS^{(i,j,m)}(w_{(J,K)}) \llra{\Gr(\pi_m)} \bigvee_{(J,K) \in \I_m^t/\I_m^{t-1}} \BS^{(i,j,m+1)}(w_{\pi_m(J,K)}). \]
Our goal is to show that $Z_m$ is acyclic. Notice that for any object $(J, K)$, for which $i_{l+m_{ij}-m} \notin J$, the map 
\[\Gr(\pi_m)) : \BS^{(i,j,m)}(w_{(J,K)}) \longrightarrow \BS^{(i,j,m+1)}(w_{\pi_m(J,K)}) \]
is an equivalence. Hence, the fiber of $\Gr(\pi_m)$ is detected on objects $(J, K)$ for which $i_{l+m_{ij}-m} \in J$. We decompose such objects into two types. The first type of objects are those for which the first term of $K$ is $i_{l+m_{ij}-m}$, and the second type for which the first term is not $i_{l+m_{ij}-m}$. Consider the boundary map $\partial :  \Gr_t (Z_m)  \longrightarrow  \Gr_{t-1} (Z_m)$ on objects of the first type. We see that precisely one component of this boundary maps to an object of the second type, and on that component, it is equivalent to the map on the homotopy fibers of vertical maps in a diagram of the form described in lemma \ref{elem q-equiv} and remark \ref{pushout}. Since these maps are cellular, lemma \ref{elem q-equiv} implies that the map $\partial$ on this component gives rise to an equivalence. One therefore has a retraction to $\partial$ on objects of the second type, giving rise to a stable chain homotopy. 

\medskip
\noindent
It remains to show that $\iota_m : \sBS^{(i,j,m+1)}(w_I)  \longrightarrow \pi_m^* \sBS^{(i,j,m+1)}(w_I)$ is also a quasi-equivalence. For this, let $F_t W_m$ denote the filtered $G$-spectrum representing the cofiber of $\iota_m$. On the level of associated graded, we have a cofibration induced by $\iota_m$
\[ \bigvee_{(\tilde{J},\tilde{K}) \in \I_{m+1}^t/\I_{m+1}^{t-1}} \BS^{(i,j,m+1)}(w_{(\tilde{J},\tilde{K})}) \longrightarrow \bigvee_{(J,K) \in \I_m^t/\I_m^{t-1}} \pi_m^* \BS^{(i,j,m+1)}(w_{(J,K)}) \longrightarrow \Gr_t (W_m). \]
Invoking theorem \ref{subdiv} we see that the first map admits a retraction, and we are left with the identification of $\Gr_t(W_m)$ with the complementary summand in the middle term. We express this summand as
\[   \bigvee_{\{ (J, K) \in A \}} \BS^{(i,j,m+1)}(w_{(J, K)},  \]
where the set $A \subseteq \I_m^t/\I_m^{t-1}$ denotes the collection of pairs $(J,K)$ for which $i_{l+m_{ij}-m} \in J$ and $K$ can be augmented in $\mathcal{O}_{red}^{(i,j,m)}$ by adding terms on the left. 
Notice that the collection of objects in $A$ come in two types determined by the sequence $K$. The first type are the ones where the first term of $K$ begins with $i_{l+m_{ij}-m+1}$ (or $K$ is the empty sequence), and the second type being the ones where the first term is $i_{l+m_{ij}-m}$. The boundary $\partial$ pairs up the terms of the first type with those of the second, and is an equivalence between these terms. In particular, we obtain the chain homotopy as before given by the inverse of $\partial$ on these terms. It follows that the cofiber of $\iota_m$ is acyclic. This completes the proof of theorem \ref{braid rel}.

\section{Properties of strict broken symmetries: Inverse relation and reflexivity} \label{Inverse}

\medskip
\noindent
In this section, we address the inverse relation and the property of reflection. In the former case, one contracts successive terms in an indexing sequence $I$ of the form $\{ \ldots, -i, i, \ldots \}$ or $\{ \ldots, i, -i, \ldots \}$ without changing the equivariant homotopy type up to possible quasi-equivalence, suspension and shift. In the latter case, one reflects the indexing sequence $I = \{\epsilon_1 i_1, \ldots, \epsilon_k i_k \}$ to the form $\{ \epsilon_k i_k, \epsilon_{k-1} i_{k-1}, \ldots, \epsilon_1 i_1 \}$ without changing the equivariant homotopy type. 

\medskip
\noindent
We begin with the following theorem.

\medskip
\begin{thm} \label{inv rel}
Let $I_{\pm}$ and $I_{\mp}$ denote indexing sequences of the form 
\[ I_{\pm} = \{ \epsilon_1 i_1, \ldots, \epsilon_l i_l, i, -i, \epsilon_{l+3} i_{l+3}, \ldots \epsilon_k i_k \}, \quad I_{\mp} = \{ \epsilon_1 i_1, \ldots, \epsilon_l i_l, -i, i, \epsilon_{l+3} i_{l+3}, \ldots \epsilon_k i_k \}, \]
then there exists an elementary quasi-equivalence between the filtered $G$-spectra $\sBC(w_{I_{\pm}})$ or $\sBC(w_{I_{\mp}})$ and the shifted spectrum $\Sigma \sBC(w_{I_{red}})[-1]$ (see \ref{shift} for the definition of shift), where 
\[ \quad I_{red} = \{  \epsilon_1 i_1, \ldots, \epsilon_l i_l, \epsilon_{l+3} i_{l+3}, \ldots \epsilon_k i_k \}.   \]
\end{thm}

\medskip
\noindent
The proof of the above theorem will rest on the following two claims

\medskip
\begin{claim} \label{inv1}
The inclusion map $\iota_i : T_+ \longrightarrow {G_i}_+$ induces a $T \times T$ equivariantly split injection
\[ \iota_i : S^{-\zeta_i} \wedge {G_i}_+ = T_+ \wedge_T  (S^{-\zeta_i} \wedge {G_i}_+) \longrightarrow {G_i}_+ \wedge_T (S^{-\zeta_i} \wedge {G_i}_+). \]
\end{claim}
\begin{proof}
We simply need to furnish an equivariant retraction. Recall that $S^{-\zeta_i}$ was the restriction to $T$ of a $G_i$ representation. In particular, the $T$-action on $S^{-\zeta_i}$ extends to a $G_i$-action. The retraction we seek is given by the left $G_i$-action on $S^{-\zeta_i} \wedge {G_i}_+$
\[ \mu : {G_i}_+ \wedge_T (S^{-\zeta_i} \wedge {G_i}_+) \longrightarrow S^{-\zeta_i} \wedge {G_i}_+. \]
\end{proof}

\begin{claim} \label{inv2}
The pinch map $\pi_i : S^{-\zeta_i} \wedge {G_i}_+ \longrightarrow T_+$ of claim \ref{PT1} induces a $T \times T$ equivariantly split surjection
\[ \pi_i : {G_i}_+ \wedge_T (S^{-\zeta_i} \wedge {G_i}_+) \longrightarrow {G_i}_+ \wedge_T T_+ = {G_i}_+ .\]
\end{claim}
\begin{proof}
It is easy to see that the fiber of $\pi_i$ is given by the spectrum ${G_i}_+ \wedge (S^{-\zeta_i} \wedge_T {T\sigma_i}_+) $, where $\sigma_i \in G_i$ is any lift of the Weyl generator by the same name. It remains to construct  an equivariant retraction from ${G_i}_+ \wedge_T (S^{-\zeta_i} \wedge {G_i}_+)$ to ${G_i}_+ \wedge_T (S^{-\zeta_i} \wedge {T\sigma_i}_+)$. This retraction $r_i$ may be defined as follows:
\[ r_i : {G_i}_+ \wedge_T (S^{-\zeta_i} \wedge {G_i}_+) \longrightarrow {G_i}_+ \wedge_T (S^{-\zeta_i} \wedge {T\sigma_i}_+), \quad \quad (g, \lambda, h) \longmapsto (gh\sigma_i^{-1}, (\sigma_i h^{-1})_\ast \lambda, {\sigma_i}).\]
\end{proof}

\begin{remark} \label{retract}
It is straightforward to check that the composite map given by the inclusion $\iota_i$ followed by the retraction $r_i$ is an equivalence
\[ r_i \circ \iota_i : S^{-\zeta_i} \wedge {G_i}_+ \llra{\cong} {G_i}_+ \wedge_T (S^{-\zeta_i} \wedge {T\sigma_i}_+).\]
\end{remark}

\noindent
In proving theorem \ref{inv rel} we will only address the case of $I := I_{\pm}$, the other case being similar. First, let us consider the homotopy colimit $\hocolim_{J \in \I^t} \BC(w_J)$. We may decompose $\I^t$ into subcategories so that this homotopy colimit may be expressed as the homotopy colimit over the following diagram

\begin{align*}
      \xymatrix{
      &   \hocolim_{J \in \I^t_{red}} \BC(w_{J \cup \{i\}})    &  \\
        \hocolim_{J \in \I^{t-1}_{red}} \BC(w_{J})   \ar[ur]    \ar[d]          & \ar[l] \hocolim_{J \in \I^{t-2}_{red}} \BC(w_{J \cup \{-i\}}) \ar[r] \ar[u] \ar[dl] \ar[d] \ar[dr] &    \hocolim_{J \in \I^{t-1}_{red}} \BC(w_{J \cup \{i, -i\}})  \ar[ul]  \ar[d] \\
     \BC(w_{I_{red}^+}) &  \ar[l] \BC(w_{I_{red}^+ \cup \{-i\}}) \ar[r] &  \BC(w_{I_{red}^+ \cup \{i,-i\}}) 
      }
\end{align*}

\medskip
\noindent
Now the entire diagram fibers over $\BC(w_{I^+}) = \BC(w_{I_{red}^+ \cup \{ i \}})$. In particular, we may express $\BC(w_{I^+}) = \BC(w_{I_{red}^+ \cup \{ i \}})$ as a colimit over a similar diagram 

\begin{align*}
      \xymatrix{
      &   \BC(w_{I_{red}^+ \cup \{i\}})    &  \\
       \quad \quad \quad \BC(w_{I_{red}^+})   \ar[ur]    \ar[d]^{=}          & \ar[l] \BC(w_{I_{red}^+ \cup \{-i\}}) \ar[r] \ar[u] \ar[dl] \ar[d]^{=} \ar[dr] &    \BC(w_{I_{red}^+ \cup \{i, -i\}})  \ar[ul]  \ar[d]^{=} \\
     \BC(w_{I_{red}^+}) &  \ar[l] \BC(w_{I_{red}^+ \cup \{-i\}}) \ar[r] &  \BC(w_{I_{red}^+ \cup \{i,-i\}}) 
      }
\end{align*}

\medskip
\noindent
Now consider the fibration
\[ \hocolim_{J \in \I^t} \BC(w_J) \llra{\pi} \BC(w_{I^+}) \longrightarrow F_t \, \sBC(w_I) . \]
We may express $F_t \, \sBC(w_I)$ as a homotopy colimit of the pointwise cofibers, denoted as $\widetilde{\BC}(w_J)$, of the above two diagrams:

\begin{align*}
      \xymatrix{
       &    \hocolim_{J \in \I^t_{red}} \widetilde{\BC}(w_{J \cup \{i\}})     &   \\
      F_{t-1} \, \sBC(w_{I_{red}})   \ar[ur]    \ar[d]          & \quad \ar[l] \hocolim_{J \in \I^{t-2}_{red}} \widetilde{\BC}(w_{J \cup \{-i\}}) \ar[r] \ar[u] \ar[dl] \ar[d] \ar[dr] &    \hocolim_{J \in \I^{t-1}_{red}} \widetilde{\BC}(w_{J \cup \{i, -i\}})  \ar[ul]  \ar[d] \\
    \ast   & \quad  \ar[l] \ast \ar[r] &  \ast
      }
\end{align*}

\medskip
\noindent
Now, using claim \ref{inv2}, it is easy to see that the following map in the above diagram
\[ F_{t-1} \, \sBC(w_{I_{red}}) \longrightarrow  \hocolim_{J \in \I^t_{red}} \widetilde{\BC}(w_{J \cup \{i\}}) \]
 lifts to $ \hocolim_{J \in \I^{t-1}_{red}} \widetilde{\BC}(w_{J \cup \{i, -i\}})$. Using this lift (namely by adding the negative of the lift), we may construct an inclusion of the following pushout representing $\Sigma F_{t-1} \, \sBC(w_{I_{red}})$
 \[ \ast \longleftarrow F_{t-1} \, \sBC(w_{I_{red}})  \longrightarrow \ast \]
 into the above homotopy colimit diagram representing $F_t \, \sBC(w_I)$. The cokernel of this inclusion is the homotopy colimit given by the $G$-spectrum $Z$, with $F_t Z$ defined as the homotopy colimit of a diagram described below

\medskip
\begin{align*}
      \xymatrix{
  {\quad \quad \quad \quad} &  &    \hocolim_{J \in \I^t_{red}} \widetilde{\BC}(w_{J \cup \{i\}})     &    \\
 \ast \ar[d] \ar[urr]  &     \quad       & \ar[ll] \hocolim_{J \in \I^{t-2}_{red}} \widetilde{\BC}(w_{J \cup \{-i\}}) \ar[r] \ar[u] \ar[dll] \ar[d] \ar[dr] &    \hocolim_{J \in \I^{t-1}_{red}} \widetilde{\BC}(w_{J \cup \{i, -i\}})  \ar[ul]  \ar[d] \\
   \ast  &   \quad &  \ar[ll] \ast \ar[r] &  \ast
      }
\end{align*}

\medskip
\noindent
Consider the associated graded complex of $Z$. All the above maps are trivial on the level of associated graded and so we see that 
\[ \Gr_t (Z) = \bigvee_{J \in \I^t_{red}} \BC(w_{J \cup \{i\}}) \vee \bigvee_{J \in \I^{t-1}_{red}} \BC(w_{J \cup \{i, -i\}}) \vee \bigvee_{J \in \I^{t-2}_{red}} \BC(w_{J \cup \{-i\}}). \]
Using claims \ref{inv1}, \ref{inv2} and remark \ref{retract}, it is easy to see that the nontrivial horizontal map in the above diagram admits an objectwise retraction and the nontrivial slanted map is objectwise split. It follows that the differential on $\Gr_t(Z)$ pairs up these split summands, from which it follows that the filtered spectrum $Z$ is acyclic. We therefore deduce that $\sBC(w_{I_{\pm}})$ is quasi-equivalent to $\Sigma \sBC(w_{I_{red}})[-1]$ as we wanted to show. The above argument completes the proof of theorem \ref{inv rel} and establishes theorem \ref{main1}. 

\medskip
\begin{remark} \label{q-equiv}
Theorem \ref{main1} tells us that $\sBC(w)$ depends only on the braid element $w$ up to quasi-equivalence. However, given a $G$-equivariant cohomology theory $\E_G$, it is not immediate that the cochain complex $\E^*_G(\Gr_t(\sBC(w))$ is well defined up to quasi-isomorphism. By invoking claim \ref{q-isom}, this would indeed be the case if we could check that each map in the zig-zag used to establish the proof of theorem \ref{main1}, is either injective or surjective on the level of $\E^*_G(\Gr)$. Analyzing the proof of theorems \ref{braid rel} and \ref{inv rel} (that feed into the proof of theorem \ref{main1}), one observes that this condition is purely formal for most maps since the associated graded complex for these map splits. The only ones for which this condition needs to be verified in cohomology are the following elementary quasi-equivalences of filtered $G$-spectra in the proof of theorem \ref{braid rel}, for any pair of indices $(i,j)$, and for $1 < m < m_{ij}$
\[ \pi_m : \sBS^{(i,j,m)}(w_I) \longrightarrow \pi_m^\ast \sBS^{(i,j,m+1)}(w_I). \]
\end{remark}

\medskip
\begin{remark} \label{fundamental}
Let us observe that the proofs of invariance under braid and inversion relations given in sections \ref{Braid} and \ref{Inverse} actually hold for the underlying $T \times T$-spectra $\sBC_T(w_I)$ before we induce up to $\SU(r)$. The invariance under these relations consequently also holds for the (strict) ``Bott-Samelson" spectra $\sBC(w_I) \wedge_T S^0$, where we have taken orbits under the right $T$-action. The spectrum $\sBC(w_I) \wedge_T S^0$ is a filtered equivariant spectrum that represents an equivariant filtered homotopy type for the complexes studied by Rouquier in \cite{R,R2}. We revisit the Bott-Samelson spectra in (\cite{Ki2} see in particular theorem 2.7). \end{remark}

\smallskip
\noindent
We end this section with an additional property of the invariant $\sBC(w_I)$ of interest. 

\smallskip
\begin{defn} (The reflected poset $\I^R$)

\noindent
Given a sequence $\, I = \{ \epsilon_1 i_1, \epsilon_2 i_2, \ldots, \epsilon_k i_k \}$, we define its reflection 
\[ I^R = \{ \epsilon_k i_k, \epsilon_{k-1} i_{k-1}, \ldots, \epsilon_1 i_1 \}. \]
This gives rise to an isomorphism of the posets $R : 2^{I} \longrightarrow 2^{I^R}$ which restricts to an isomorphism $R : \I \longrightarrow \I^R$, where $\I^R$ is the poset of all non-terminal objects in $2^{I^R}$. Given a subset $J \in 2^{I}$, its image under $R$ is an element in $2^{I^R}$ denoted by $J^R$. 
\end{defn}

\noindent
With the above definition in place, let us prove the reflexive property: 

\medskip
\begin{thm}\label{Mirror} The functors $\BC(w_{J})$ and $\BC(w_{J^R}) \circ R$ are equivalent. In particular, $R$ induces a levelwise (honest) equivalence of filtered $G$-spectra
\[ R_t : F_t \, \sBC(w_{I}) \llra{\simeq} F_t \, \sBC(w_{I^R}), \quad t \geq 0. \]
\end{thm}
\begin{proof}
We require a natural equivalence between the $G$-spectra $\BC(w_{J})$ and $\BC(w_{J^R})$. Let $J = \{\epsilon_{i_{j_1}} i_{j_1}, \ldots, \epsilon_{i_{j_s}} i_{j_s} \}$ be an element in $2^{I}$. Recall from definition \ref{BC} that 
\[  \BC(w_{J}) : G_+ \wedge_T (H_{i_{j_1}} \wedge_T H_{i_{j_2}} \wedge_T \ldots  \wedge_T H_{i_{j_s}}),  \, \, \mbox{where} \, \, H_i = S^{-\zeta_i} \wedge \, {G_i}_+,  \, \, \mbox{if} \, \, \epsilon_i = -1, \, \,  H_i = {G_i}_+ \, \,  \mbox{else}. \]
Let us first consider the case of a positive braid $w_{I}$, so that $H_i = {G_i}_+$ for all $j$. In that case, we define $R$ on the underlying topological space by
\[ R [(g, g_{i_{j_1}}, \ldots, g_{i_{j_s}})] = [(g (g_{i_{j_1}} \ldots g_{i_{j_s}}), g_{i_{j_s}}^{-1}, \ldots, g_{i_{j_1}}^{-1})], \quad (i_{j_s}, \ldots, i_{j_1}) = J^R. \]
We may express the above map as $R = \mu \wedge_T (R_{i_{j_s}} \wedge_T R_{i_{j_{s-1}}} \ldots \wedge_T R_{i_{j_1}})$, where $\mu$ is induced by the multiplication map: 
\[ \mu : G \times G_{i_{j_1}} \times \cdots \times G_{i_{j_s}} \longrightarrow G, \quad (g, g_{i_{j_1}}, \ldots, g_{i_{j_s}}) \longmapsto g g_{i_{j_1}}, \ldots g_{i_{j_s}},  \]
and $R_i : G_i \rightarrow G_i$ is the inversion map. 

\noindent
We now extend the above description of $R$ to the case of an arbitrary braid $w_{I}$. To begin, let us observe that $S^{-\zeta_i} \wedge {G_i}_+$ admits a map of the form $S^{-\zeta_i} \wedge {G_i}_+ \longrightarrow S^{-\zeta_i} \wedge {G_i}_+ \wedge {G_i}_+$ given by performing the diagonal on the last factor. Therefore, for an arbitrary braid $w_{I}$, we have a map of the form 
\[ \mbox{D} : G_+ \wedge (H_{i_{j_1}} \wedge H_{i_{j_2}} \wedge \ldots  \wedge H_{i_{j_s}}) \longrightarrow G_+ \wedge ((H_{i_{j_1}} \wedge {G_{i_{j_1}}}_+) \wedge (H_{i_{j_2}} \wedge {G_{i_{j_2}}}_+) \wedge \ldots  \wedge (H_{i_{j_s}} \wedge {G_{i_{j_s}}}_+)). \]
The map $\mu$ defined above can now be invoked to obtain. 
\[ \mu : G_+ \wedge ((H_{i_{j_1}} \wedge {G_{i_{j_1}}}_+) \wedge (H_{i_{j_2}} \wedge {G_{i_{j_2}}}_+) \wedge \ldots  \wedge (H_{i_{j_s}} \wedge {G_{i_{j_s}}}_+)) \longrightarrow G_+ \wedge (H_{i_{j_1}} \wedge H_{i_{j_2}}  \wedge \ldots  \wedge H_{i_{j_s}} ). \]
Hence, we have a self-map $M$ of $G_+ \wedge (H_{i_{j_1}} \wedge H_{i_{j_2}} \wedge \ldots  \wedge H_{i_{j_s}})$ given by $M = \mu \circ \mbox{D}$. The map $R_i$ also extends to a spectrum of the form $S^{-\zeta_i} \wedge {G_i}_+$ as follows. Let $\llbracket -1\rrbracket$ denote the involution on $S^{-\zeta_i} = \Map(S^{\mathfrak{g}_i}, S^r)$ induced by conjugation with the antipode maps on $\mathfrak{g}_i$ and $\R^r$. We define $R_i$ on $S^{-\zeta_i} \wedge {G_i}_+$ to be the involution given by smashing $Ad(g_i^{-1})_\ast \circ \llbracket -1\rrbracket$ on $S^{-\zeta_i}$, with the inversion map on ${G_i}_+$. 

\medskip
\noindent
We now define the natural equivalence $R$ from $\BC(w_{J})$ to $\BC(w_{J^R})$ as 
\[ R : \BC(w_{J}) \longrightarrow \BC(w_{J^R}), \quad R := (R_{i_{j_s}} \wedge_T R_{i_{j_{s-1}}} \wedge_T \ldots \wedge_T R_{i_{j_1}}) \circ M. \]
It is straightforward to check from the construction that $R$ is well-defined and indeed a natural equivalence of functors. 
\end{proof}

\medskip
\noindent
\section{Properties of strict broken symmetries: $G = \SU(r)$ and the Markov 2 property} \label{Markov2} 

\medskip
\noindent
In this section, we specialize to the case of the compact Lie group $G=\SU(r)$, whose braid group $\Br(r)$ is the classical braid group on $r$-strands generated by the elementary classical braids $\sigma_1, \ldots, \sigma_{r-1}$.

\medskip
\noindent
The Markov 2 property studies the effect of taking a braid word in $r$-strands, and extending it to a braid word in $(r+1)$-strands by augmenting it by the generator $\sigma_r \in \Br(r+1)$. 
Since we will compare the spectra of broken symmetries for $\SU(r)$ and $\SU(r+1)$, we see that the Markov 2 property introduces a stabilization in the strands. In order to be keep track of the number of strands, let us set some notation. For $i \leq r$, we will use the notation $G_i^r \subseteq \SU(r+1)$ to be the unitary form (of rank $r+1$) in the reductive Levi subgroup with roots $\pm \alpha_i$. We will continue to use the notation $G_i \subseteq \SU(r)$ for the unitary form of rank $r$. Notice that for $i < r$, these subgroups of $\SU(r)$ and $\SU(r+1)$ are related via a block decomposition $G_i^r = G_i \times S^1$, where $S^1$ denotes the last factor of the product decomposition of the standard maximal torus $T^{r+1} \subset \SU(r+1)$. 

\medskip
\noindent
Let $\Delta_r \subset T^{r+1}$ denote the centralizer of the final simple root $\alpha_r$. More precisely, $\Delta_r$ is the subgroup $T^{r-1} \times \Delta$, where $\Delta$ is the diagonal circle in the last two standard factors of $T^{r+1}$. We may re-express $T^{r+1}$ as $\Delta_r \times S^1$, with $S^1$ being identified with the last factor in the standard decomposition of $T^{r+1}$. By construction, $\Delta_r$ centralizes the group $G^r_r$.

\medskip
\noindent
The goal of this section is to establish the following two theorems. 

\medskip
\begin{thm} \label{M2a}
Let $I = \{\epsilon_1 i_1, \cdots, \epsilon_k i_k \}$ denote a sequence that offers a presentation for a braid element $w \in \Br(r)$, and let $I(r)$ denote the sequence obtained by augmenting $I$ by the index $i_{k+1} = r$. In other words, $I(r)$ is a presentation for the braid element $w \sigma_r \in \Br(r+1)$. Then there is an elementary quasi-equivalence of $\SU(r+1)$-spectra
\[ \sBC(w_{I(r)}) \longrightarrow \SU(r+1)_+ \wedge_{\Delta_r} \Sigma^2 \sBC_{T^r}(w_I), \]
where we think of $\Delta_r$ as a map $\Delta_r : T^r \longrightarrow T^{r+1}$, $(t_1, \ldots, t_{r-1}, t_r) \longmapsto (t_1, \ldots, t_{r-1}, t_r, t_r)$.
\end{thm}

\medskip
\begin{thm} \label{M2b}
Let $I = \{\epsilon_1 i_1, \cdots, \epsilon_k i_k \}$ denote a sequence that offers a presentation for a braid element $w \in \Br(r)$, and let $I(-r)$ denote the sequence obtained by augmenting $I$ by the index $i_{k+1} = -r$. In other words, $I(-r)$ is a presentation for the braid element $w \sigma_r^{-1} \in \Br(r+1)$. Then there is an elementary quasi-equivalence of $\SU(r+1)$-spectra
\[  \sBC(w_{I(-r)}) \longrightarrow \SU(r+1)_+ \wedge_{\Delta_r} \Sigma^{-1} \sBC_{T^r}(w_I). \]
\end{thm}

\medskip
\noindent
The proof of theorem \ref{M2a} rests on the following claim

\medskip
\begin{claim} \label{M2cof}
Let $J \subseteq I$ denote a subsequence $J = \{\epsilon_{j_1} i_{j_1}, \ldots, \epsilon_{j_s} i_{j_s}\}$. Regarding $J$ as a subsequence of $I(r)$, let $J(r) \subseteq I(r)$ denote the sequence $J \cup \{i_{k+1} \}$. Then there is a cofibration of $\SU(r+1)$-equivariant spectra induced by the inclusion $J \subset J(r)$
\[ \BC(w_J) \longrightarrow \BC(w_{J(r)}) \longrightarrow \SU(r+1)_+ \wedge_{\Delta_r} \Sigma^2 \BC_{T^r}(w_J). \]
\end{claim}

\begin{proof}
Let $T^{r+1} \subset G^r_r \subset \SU(r+1)$ denote the inclusion of the standard maximal torus $(S^1)^{\times r+1}$. Notice that we have a cofibration of $T^{r+1} \times T^{r+1}$-equivariant spectra
\begin{equation} \label{eq1} T^{r+1}_+ {\longrightarrow G^r_r}_+ \longrightarrow S^{\zeta_r} \wedge (T^{r+1} \sigma_r)_+, \end{equation}
where $\sigma_r$ is the permutation matrix in $\SU(r+1)$ that permutes the last two standard coordinates of $T^{r+1}$, so that $T^{r+1} \sigma_r$ is a $T^{r+1} \times T^{r+1}$-space abstractly isomorphic to $T^{r+1}$, with the right $T^{r+1}$-action being twisted by $\sigma_r$. As before, $S^{\zeta_r}$ is the compactification of the root space representation of the root $\alpha_r$, and is given a trivial right $T^{r+1}$ action. 

\medskip
\noindent
Smashing equation (\ref{eq1}), $T^{r+1}$-equivariantly, with the spectra $\BC_{T^{r+1}}(w_J)$, we get a cofibration
\begin{equation}\label{eq2} \BC(w_J) \longrightarrow \BC(w_{J(r)}) \longrightarrow \SU(r+1)_+ \wedge_{T^{r+1}} (H^r_{i_{j_1}} \wedge_{T^{r+1}} \cdots \wedge_{T^{r+1}} H^r_{i_{j_s}} \wedge S^{\zeta_r} \wedge {\sigma_r}_+). \end{equation}
Decomposing $T^{r+1}$ as $\Delta_r \times S^1$, and observing that the right action of $\Delta_r$ fixes the spectrum $S^{\zeta_r}$ and commutes with $\sigma_r$, we may express the above cofiber as
\[ (\SU(r+1)_+ \wedge_{\Delta_r}  \wedge (H^r_{i_{j_1}} \wedge_{T^{r+1}} \cdots \wedge_{T^{r+1}} H^r_{i_{j_s}}) \wedge S^{\zeta_r} \wedge {\sigma_r}_+) \wedge_{S^1} S^0. \]
Recall that the $S^1$-action on $(H^r_{i_{j_1}} \wedge_{T^{r+1}} \cdots \wedge_{T^{r+1}} H^r_{i_{j_s}}) \wedge {\sigma_r}_+$ is by endpoint conjugation. Incorporating the twisting by $\sigma_r$ on the right allows us to identify the above $S^1$-spectrum with $(H^r_{i_{j_1}} \wedge_{T^{r+1}} \cdots \wedge_{T^{r+1}} H^r_{i_{j_s}})$, with the $S^1$-action given by twisting the conjugation action by $\sigma_r$ on the right hand side. This twisted conjugation action can be identified with the standard right multiplication action of the conjugate diagonal subgroup $S^1 \cong \overline{\Delta} \subseteq T^{r+1}$ consisting of elements of the form $(x^{-1}, x)$ in the last two factors. Recall that for $i \in I$, we have a block decomposition $G^r_j = G_r \times S^1$. It follows that all the spectra $H^r_i$ that occur above are free $S^1$-spectra of the form $H_i \wedge S^1_+$, where $H_i$ denotes the corresponding spectra when $J$ is seen as a subset of $I$. We may therefore express $(H^r_{i_{j_1}} \wedge_{T^{r+1}} \cdots \wedge_{T^{r+1}} H^r_{i_{j_s}})$ as $(H_{i_{j_1}} \wedge_{T^r} \cdots \wedge_{T^r} H_{i_{j_s}} \wedge S^1_+)$. In particular, the cofiber of equation \ref{eq2} can be identified with
\[ (\SU(r+1)_+ \wedge_{\Delta_r} (H_{i_{j_1}} \wedge_{T^r} \cdots \wedge_{T^r} H_{i_{j_s}}) \wedge S^{\zeta_r} )\wedge_{\overline{\Delta}} S^1_+). \]
Since the $\overline{\Delta}$-action is free on the $S^1$-factor, we may drop the free $S^1$-factor and identify the above spectrum with
\begin{equation} \label{eq3} \SU(r+1)_+ \wedge_{\Delta_r} \Sigma^2 (H_{i_{j_1}} \wedge_{T^r} \cdots \wedge_{T^r} H_{i_{j_s}} ). \end{equation}
Putting equation \ref{eq2} and the identification \ref{eq3} together, gives rise to the cofibations of $\SU(r+1)$-spectra that we seek
\[ \BC(w_J) \longrightarrow \BC(w_{J(r)}) \longrightarrow \SU(r+1)_+ \wedge_{\Delta_r} \Sigma^2 \BC_{T^r}(w_J). \]
\end{proof}

\noindent
Let us use the above claim to prove theorem \ref{M2a}. Let us first recall the categories used in defining the spectra $\sBC(w_{I(r)})$. 

\medskip
\begin{defn} (The poset $\I(r)$ and the functor $\BC^r(w_J)$ ) \label{category L}

\noindent
Let $\I(r) \subset 2^{I(r)}$ denote the poset subcategory of subsets in $I(r)$ that do not contain the terminal object. Consider the functor $\BC^r$ from $\I(r)$ to $\SU(r+1)$-spectra that sends $J \in \I(r)$ to $\BC(w_{J \cap I})$. It is clear that $\BC(w_J) = \BC^r(w_J)$ if $J \subseteq I$. In particular, the above functor is a natural extension of the functor $\BC$ on $2^I$. Let us also observe that one has a canonical natural transformation $\mathcal{T} : \BC^r \longrightarrow \BC$ induced by the inclusions $J\cap I \subset J(r)$. 
\end{defn}

\medskip
\noindent
Now consider the following commutative diagram
\[
\xymatrix{
\hocolim_{J \in \I(r)^t} \BC^r(w_J)  \ar[d]^{\hocolim \mathcal{T}} \ar[r] & \BC^r(w_{I^+}) \ar[d]  \\
\hocolim_{J \in \I(r)^t} \BC(w_J) \ar[d] \ar[r] & \BC(w_{I(r)^+}) \ar[d] \\
\hocolim_{J \in \I^t} \SU(r+1)_+ \wedge_{\Delta_r} \Sigma^2 \BC_{T^r} (w_J)  \ar[r] & \SU(r+1)_+ \wedge_{\Delta_r} \Sigma^2 \BC_{T^r}(w_{I^+}).
}
\]
\noindent
It is clear from claim \ref{M2cof} that the right vertical sequence is a cofibration. Let us notice that the left vertical sequence is also a cofibration. To see this, recall that the functor $\BC(w_J)$ agrees with the functor $\BC^r$ on the full sub-category generated by $J \in \I(r)^t$ that do not contain $i_{r+1}$. This sub-category has a terminal object $I$. In particular, the cofiber of $\hocolim \mathcal{T}$ is detected on the full sub-category of objects $J$ containing $i_{r+1}$. This category is equivalent to $\I^t$, and one may identify the cofiber of the top left vertical map with $\hocolim_{J \in \I^t} \SU(r+1)_+ \wedge_{\Delta_r} \Sigma^2 \BC_{T^r} (w_J)$ using claim \ref{M2cof}. This shows that the left vertical sequence is a cofibration. Taking horizontal cofibers of the above diagram gives rise to a cofibration of filtered $\SU(r+1)$-spectra
\[ s\BC^r(w_{I(r)}) \llra{s\mathcal{T}} s\BC(w_{I(r)}) \longrightarrow \SU(r+1)_+ \wedge_{\Delta_r} \Sigma^2 \sBC_{T^r}(w_I), \]
where the filtered $\SU(r+1)$-spectrum $s\BC^r(w_{I(r)})$ is defined to have filtrates $F_t \, s\BC^r(w_{I(r)})$ given by the cofiber of the top horizontal map. It remains to show that $s\BC^r(w_{I(r)})$ is acyclic. From the definition of $s\BC^r(w_{I(r)})$, the associated graded of the filtered $\SU(r+1)$-spectrum $s\BC^r(w_{I(r)})$ is easily computed to be
\[ \Gr_t(s\BC^r(w_{I(r)})) = \Gr_{t-1}(\sBC(w_I)) \vee \Gr_t(\sBC(w_I)), \]
with the differential identifying the obvious summands. The null homotopy is straightforward to construct, completing the proof of theorem \ref{M2a}. 

\medskip
\noindent
The proof of theorem \ref{M2b} is similar to the above and rests on the following claim similar to claim \ref{M2cof}. We sketch the argument below, leaving the details to the reader

\medskip
\begin{claim}
Let $J \subseteq I$ denote a subsequence $J = \{\epsilon_{j_1} i_{j_1}, \ldots, \epsilon_{j_s} i_{j_s}\}$. Regarding $J$ as a subsequence of $I(-r)$, let $J(-r) \subseteq I(-r)$ denote the sequence $J \cup \{-i_{k+1} \}$. Then there is a cofibration of $\SU(r+1)$-equivariant spectra induced by the map $J(-r) \rightarrow J$
\[ \BC(w_{J(-r)}) \longrightarrow \BC(w_J) \longrightarrow \SU(r+1)_+ \wedge_{\Delta_r} \Sigma^{-1} \BC_{T^r}(w_J). \]
\end{claim}

\medskip
\noindent
The proof of this claim is formally the same as that of claim \ref{M2cof} and starts with the cofibration sequence of $T^{r+1} \times T^{r+1}$ spectra induced by the inclusion $T^{r+1} \sigma_r \subseteq G^r_r$
\[ S^{-\zeta_r} \wedge (T^{r+1} \sigma_r)_+ \longrightarrow S^{-\zeta_r} \wedge {G_r^r}_+ \longrightarrow T^{r+1}_+ \longrightarrow \Sigma S^{-\zeta_r} \wedge (T^{r+1} {\sigma_r})_+. \]
We leave it to the reader to complete the argument along the lines of claim \ref{M2cof}. 

\medskip
\noindent
The proof of theorem \ref{M2b} is now very similar to that of theorem $\ref{M2a}$. One begins by defining a functor $\BC^{-r}$ from $\I(-r)$ to $\SU(r+1)$-spectra that sends $J \in \I(-r)$ to $\BC(w_{J \cup \{-i_{k+1}\}})$. It is clear that $\BC(w_J) = \BC^{-r}(w_J)$ if $-i_{k+1} \in J$. The rest of the proof follows from chasing a diagram similar to the one described in the proof of theorem $\ref{M2a}$. We leave it to the reader to complete the proof.

\section{The invariant $\sBC(L)$ of (framed) links and the Galois symmetry} \label{one}

\medskip
\noindent
By now it it clear that the invariant $\sBC(w_I)$ that has been studied in the previous few sections enjoys several important properties. Theorem \ref{M1} shows that $\sBC(w_I)$ is equivalent to its cyclic permutation $\sBC(w_{\overline{I}})$, which is known as the Markov 1 property. In theorem \ref{main1} we showed that $\sBC(w_I)$ did not depend on the indexing sequence $I$ used to present the braid word $w$. And finally, in the case $G = \SU(r)$, theorems \ref{M2a} and \ref{M2b} demonstrated that $\sBC(w_I)$ satisfied an interesting variant of the Markov 2 property. 

\medskip
\noindent
Of particular importance is the Markov 2 property, which we turn our attention to for the moment. Recall that by theorems \ref{M2a} and \ref{M2b}, we have elementary quasi-equivalences
\[ \sBC(w_{I(r)}) \longrightarrow  \SU(r+1)_+ \wedge_{\Delta_r} \Sigma^2 \sBC_{T^r}(w_I) , \, \, \]
\[ \sBC(w_{I(-r)}) \longrightarrow \SU(r+1)_+ \wedge_{\Delta_r} \Sigma^{-1} \sBC_{T^r}(w_I), \]
where we consider $\Delta_r$ as a map $\Delta_r : T^r \longrightarrow T^{r+1}$, $(t_1, \ldots, t_{r-1}, t_r) \longmapsto (t_1, \ldots, t_{r-1}, t_r, t_r)$, 
and $I(\pm r)$ is the multiindex representing the braid word $w_I \sigma_r^{\pm 1}$ in $\Br(r+1)$. 

\medskip
\noindent
The difference in the number of suspensions in these two equivalences is easily corrected when we normalize and consider the invariant $\sNBC(w)$ of definition \ref{stableLI}. A more subtle issue is that the invariant $\sBC(w_{I(\pm r)})$ is equivalent to the spectrum obtained by inducing the $T^r$-spectrum $\sBC_{T^r}(w_I)$ to a $\SU(r+1)$ spectrum, along a {\em non-standard} copy of the torus $T^r \subset \SU(r+1)$ given by $\Delta_r$. This observation points to a canonical enhancement of $\sBC(w)$ to a {\em framed} link invariant. The following definition describes how one may achieve this. 

\medskip
\begin{defn} \label{redef} (Framed enhancement of $\sBC(w_I)$)

\noindent
Let $T^w \subset T^r$ denote the sub torus invariant under the permutation in $\Sigma_r$ underlying $w_I$, and let $\mbox{X}^*(T^w)$ denote the character lattice of $T^w$. Note that the factors of the torus $T^w$ are canonically in bijection with the components of the link obtained in closing the braid $w_I$. In particular, elements of $\mbox{X}^\ast(T^w)$ parametrize all possible framings of this link. 

\medskip
\noindent
We define the framed enhancement of the filtered $T^r$-spectrum $\sBC_{T^r}(w_I)$ (equivalently, the filtered $\SU(r)$-spectrum $\sBC(w_I)$)  to be a pair $(\sBC_{T^r}(w_I), \chi)$ (resp. $(\sBC(w_I), \chi)$), where $\chi \in \mbox{X}^*(T^w)$. 
\end{defn}

\noindent
The curious variation of the second Markov move translates into an elementary claim:

\begin{corr} \label{redef2}
Let $\Delta_r : T^r \longrightarrow T^{r+1}$ be the map described above. Then $\Delta_r$ restricts to an isomorphism $\Delta_r : T^{w} \longrightarrow T^{w(\pm r)}$, where $T^{w(\pm r)} \subseteq T^{r+1}$ is the sub torus invariant under $w_I \sigma_r^{\pm 1}$. In particular, $\Delta_r$ induces a canonical bijection $\Delta_r^\ast : \mbox{X}^\ast(T^{w(\pm r)}) \cong \mbox{X}^\ast(T^w)$ lifting to framed elementary quasi-equivalences
 \[ (\sBC(w_{I(r)}), \chi) \longrightarrow \SU(r+1)_+ \wedge_{T^{r+1}} T^{r+1}_+ \wedge_{\Delta_r} (\Sigma^2 \sBC_{T^r}(w_I), \Delta_r^\ast(\chi)), \, \, \]
\[ (\sBC(w_{I(-r)}), \chi) \longrightarrow \SU(r+1)_+ \wedge_{T^{r+1}} T^{r+1}_+ \wedge_{\Delta_r} (\Sigma^{-1} \sBC_{T^r}(w_I), \Delta_r^\ast(\chi)). \]
\end{corr}

\medskip
\begin{Conv} \label{conv}
For the sake of simplifying notation, the framing enhancement will be implicitly understood in what follows. We will supress the notation that makes explicit reference to the framing when referring to $\sBC(w_I)$. 
\end{Conv}

\medskip
\noindent
\begin{defn} \label{stableLI2} (Normalization of strict broken symmetries as a function of links)

\noindent
Given a link $L$ described by the closure of a braid word on $r$ strands, define the normalized, filtered $\SU(r)$-equivariant spectrum as described in \ref{stableLI}
\[ \sBC(L) := \sBC(w) = \Sigma^{l(w_I)} \sBC(w_I)[\varrho_I], \]
where $w \in \Br(r)$ is any braid with presentation $w_I$, that represents the link $L$. This normalization corrects for the filtration shifts and suspensions that one encounters in proving invariance under the various properties (see remark \ref{normalization}). \end{defn}

\noindent
Having verified braid invariance and invariance under Markov moves, we conclude

\medskip
\begin{thm} \label{coh}
As a function of a (framed) link $L$ underlying the closure of a braid word on $r$-strands, the filtered $\SU(r)$-spectrum $\sBC(L)$ is well-defined up to quasi-equivalence. In particular, the limiting equivariant stable homotopy type $\sBCU(L)$ is a well-defined (framed) link invariant.  
\end{thm}

\begin{remark} \label{normalization}
We make note here that the normalization of the spectrum $\sBC(L)$ we have provided in definition \ref{stableLI2} is purely topological in nature, and may differ from the standard normalization for link invariants once we identify the cohomology of $\sBC(L)$ with such an invariant. The table given in definition \ref{condM2} indicates how the topological normalization is sensitive to the various cohomology theories we will consider in the sequel. 
\end{remark}

\noindent
At this point we come to an interesting symmetry that is compatible with the constructions of the previous chapters, and consequently descends to $\sBC(L)$. This symmetry is defined by the Galois action, denoted by $\sigma$, given by complex conjugation on the spaces $\BC(w_I)$, and the bundles $\zeta_i$. Let us study this symmetry in some detail. 

\medskip
\noindent
Given a positive sequence $I$, we define the $\sigma$ action on $\BC(w_I)$ as 
\[ \sigma [(g, g_{i_1}, \ldots, g_{i_k})] := [(\overline{g}, \overline{g}_{i_1}, \ldots, \overline{g}_{i_k})], \]
where $\overline{h}$ denotes the complex conjugation of the element $h \in \SU(r)$. Since complex conjugation is an automorphism of $\SU(r)$ that preserves all the compact subgroups $G_i$, we see that $\sigma$ gives rise to an (anti linear) automorphism of the $\SU(r)$-space $\BC(w_I)$. Now recalling the definition \ref{desusp} of the Thom spectrum $S^{-\zeta_i}$ as $\Sigma^{r} S^{-\mathfrak{g}_i}$, we may define $\sigma$ to act on $S^{-\zeta_i}$ by smashing the antipode action on $\R^r$, with  the dual of the action induced by complex conjugation on $S^{\mathfrak{g}_i}$. Recall that as a real virtual $T$-representation $\zeta_i$ is isomorphic to the root space representation for the root $\alpha_i$. As such, $\sigma$ can be identified with the canonical automorphism of $\zeta_i$ that acts by complex conjugation on the root space. 

\medskip
\noindent
It follows from the above observation that $\sigma$ extends to an anti-linear automorphism of the $\SU(r)$-spectrum $\BC(w_I)$ for arbitrary words $I$. One may verify that the constructions leading to the filtered spectrum $\sBC(w_I)$ are natural with respect to $\sigma$. In particular, $\sigma$ extends to a Galois symmetry acting on $\sBC(w_I)$.

\medskip
\begin{remark} \label{Galois}
Let us explicitly describe how one normalizes the action of $\sigma$ on $\sBC(w_I)$ so that the link invariant $\sBC(L)$ is compatible with the Galois action. For this, one recalls that the normalization by the suspension $\Sigma^{l(w_I)}$ is dictated by the suspensions appearing in theorems \ref{M2a} and \ref{M2b}. On unraveling the source of suspensions, we see that one must identify the suspension $\Sigma^{l(w_I)}$ as $\Sigma^{-l_-(w_I)} \wedge \C_+^{l_-(w_I)} \wedge \C_+^{-l_+(w_I)}$, with $\sigma$ acting by complex conjugation on the last two factors, and where $\C_+^{k}$ denotes the one-point compactification of the complex vector space $\C^{k}$. 
\end{remark}

\medskip
\noindent
Now let $\E_G$ denote a family of equivariant cohomology theories indexed by the collection $G = \SU(r)$, with $r \geq 1$, endowed with natural compatiblity under restrictions 
\[ \Delta_j : \E_{\SU(r_1)} \wedge \E_{\SU(r_k)} \ldots \wedge \E_{\SU(r_m)} \llra{\cong} j^\ast \E_{\SU(m)}, \, \, \mbox{where} \, \, \,  j : \SU(r_1) \times \SU(r_2) \times \ldots \times \SU(r_k) \hookrightarrow \SU(m), \]
is an arbitrary injection of Lie groups. For instance, we may take $\E_G$ to be a Global cohomology theory \cite{Sc}. We do not assume that $\E_G$ is multiplicative for now. In the case of multiplicative theories, one will ask for some more structure (see \cite{Ki2} section 3). 

\medskip
\begin{defn} \label{condM2} (INS-Type equivariant cohomology theories)

\noindent
Given a family of equivariant cohomology theories $\{ \E_{\SU(r)}, r \geq 1\}$ as above, we call them INS-type theories if the elementary quasi-equivalences flagged in remark \ref{q-equiv} (called $\mbox{B}$-type maps), and those of \ref{M2a} and \ref{M2b} (called $\mbox{M2a}$ and $\mbox{M2b}$-type maps resp.) induce injective, null or surjective maps on the associated graded complex. For such theories, claims \ref{q-isom} and \ref{q-isom2} show that the quasi-isomorphism type of the bi-complex $\E_{\SU(r)}^s(\Gr_t(\sBC(L)))$ is well defined up to a shift in bi-degree. By claim \ref{q-isom2} and remark \ref{regrading}, we see that the shift will depend only on the number of elementary quasi-equivalences between two braid presentations of the link L that induce null maps on the associated graded complex. Below we list how the cohomology theories considered in \cite{Go, Ki2} are expected to behave under the $\mbox{B}$, $\mbox{M2a}$ and $\mbox{M2b}$-type elementary quasi-equivalences above. 
\end{defn}
\[
\begin{tabular}{|c|c|c|c|}
\hline
$\mbox{Equivariant cohomology Theory}$ & $\mbox{B-type maps}$ & $ \mbox{M2a-type maps} $&$ \mbox{M2b-type maps}$\\
\hline
$\mbox{Singular Coh. (untwisted)}$ \cite{Ki2} & $\mbox{Injective}$ & $\mbox{Null} $& $ \mbox{Null}$\\

$\mbox{Singular Coh. (twisted)}$ \cite{Go} & $\mbox{Injective}$ & $\mbox{Null}$&$ \mbox{Injective}$\\

$\mbox{K-theory (twisted)}$ \cite{Ki2} & $\mbox{Injective (\dag)}$ & $\mbox{Null (\dag)}$ &$ \mbox{Injective (\dag)} $\\
\hline
\end{tabular}
\]
$\dag$ indicates that we expect the property to hold but we have not shown it to be true. 

\medskip
\noindent
Given an INS-type equivariant cohomology theory, an obvious strategy to construct group valued link invariants from $\sBC(L)$ is to study the spectral sequence that computes the cohomology $\E^\ast_{\SU(r)}(\sBCU(L))$ by virtue of the filtration. The $E_2$-term of this spectral sequence is the cohomology of the complex $\E^\ast_{\SU(r)}(\Gr_t(\sBC(L)))$, described in example \ref{FiltBC}. 

\medskip
\begin{thm} \label{SS}
Assume that $\E_{\SU(r)}$ is a family of INS-type $\SU(r)$-equivariant cohomology theories. Then, given a link $L$ described as a closure of a braid $w$ on $r$-strands as above, one has a cohomologically graded spectral sequence converging to $\E^\ast_{\SU(r)}(\sBCU(L))$ and with $E_1$-term given by 
\[ E_1^{t,s} = \bigoplus_{J \in \I^t/\I^{t-1}} \E_{\SU(r)}^s(\BC(w_J))  \, \, \Rightarrow \, \, \E_{\SU(r)}^{s+t+l(w_I)}(\sBCU(L)). \]
The differential $d_1$ is the canonical simplicial differential induced by the functor described in definition \ref{poset}. In addition, the terms $E_q(L)$ are invariants of the link $L$ for all $q \geq 2$, upto an indeterminacy given by an overall shift in bi-degree. 
\end{thm}

\medskip
\begin{remark} \label{Galois2}
We notice that the Galois symmetry $\sigma$ descends to a symmetry of each page of the spectral sequence above. In other words, the link invariants $E_q(L)$ admit an extra symmetry given by the involution $\sigma$. 
\end{remark}

\section{The $p$-completion, \'Etale homotopy type and the Frobenius} \label{category}

\medskip
\noindent
In this very brief section, we point out a piece of algebraic structure that extends the Galois symmetry described in the previous section. This structure appears on $p$-completing our constructions. We have kept this section brief since it deviates from the general geometric flavor of the arguments we have been describing in this document. We plan to return to this structure in the future. 

\medskip
\noindent
Let us revert back to a general compact connected Lie group $G$, and assume that it is the unitary form of the complex points of a Chevalley group scheme $G_\Z$. For instance, we may take $G_\Z$ to be $\GL(r)_\Z$ in the case $G = \SU(r)$. The groups $G_i$ in definition \ref{BC+} admit $\Z$-forms given by the corresponding split reductive Levi factors. It follows that for positive sequences $I$, the spaces of broken symmetries $\BC(w_I)$ and the Bott-Samelson spaces $\BSa(w_I) := \BC(w_I)/T$ also admit $\Z$-forms $\BC_\Z(w_I)$ and $\BSa_\Z(w_I)$ resp. 

\medskip
\noindent
Now \'Etale homotopy theory \cite{F} allows us to compare the \'Etale homotopy type of schemes over the algebraic closure $\overline{\F}_q$, with the analytic space of complex points after $p$-completion at any prime $p \neq q$. It follows from these ideas that the $p$-completion of their respective suspension spectra are also equivalent. In particular, we recover the action of the Frobenius automorphism $\psi_q$ on the $p$-complete spectrum $\mbox{L}_{H\Z/p} \BC(w_I)$, and $\mbox{L}_{H\Z/p} \BSa(w_I)$ where $\mbox{L}_{H\Z/p}$ denotes Bousfield localization with respect to mod-$p$ homology. 

\medskip
\noindent
Continuing to work with positive sequences $I$, the reader can confirm that maps used to show the braid invariance of $\BC(w_I)$ or $\BSa(w_I)$ (up to quasi-equivalence) are all algebraic. In particular, the action of the Frobenius $\psi_q$ on $p$-complete spectra of broken symmetries extends to an action on the $p$-complete spectra of strict broken symmetries $\mbox{L}_{H\Z/p} \sBC(w_I)$. Since one can think of the diagram that defines $\sBC(w_I)$ as the diagram of complex points of a simplicial scheme defined over $\Z$, we conclude 

\medskip
\begin{thm} \label{Etale}
Let $I$ be a positive sequence, and let $p \neq q$ be distinct primes, then the stable $p$-completion of the \'Etale homotopy type of the simplicial scheme $\sBC_{\overline{\F}_q}(w_I)$ or the corresponding Rouquier complex $\sBSa_{\overline{\F}_q}(w_I)$, is invariant under braid relations in the presentation $w_I$ up to quasi-equivalence in the category of $p$-complete spectra endowed with an action of a (Frobenius) automorphism $\psi_q$. 
\end{thm}

\section{Appendix: Homotopy Colimits} \label{appendix}

\medskip
\noindent
We review the construction and properties of the homotopy colimit of small diagrams of $G$-spectra. The standard reference is the last few chapters of \cite{BK}, though the reader may find several very helpful modern sources (for instance \cite{D}). 

\medskip
\noindent
Let $\mathscr{F}$ denote a functor from a small category $\mathcal{C}$ to the category $G \mathscr{S}$ of $G$-spectra
\[ \mathscr{F} : \mathcal{C} \longrightarrow G \mathscr{S}. \]
Then $\mathscr{F}$ gives rise to a simplicial $G$-spectrum which we denote by $N_\bullet(\mathscr{F})$, whose $k$-simplices $N_k(\mathscr{F})$ are defined to be the $G$-spectrum
\[ N_k(\mathscr{F}) = \bigvee_{i_0 \rightarrow i_1 \rightarrow \cdots \rightarrow i_k} \mathscr{F}(i_0), \quad \mbox{for $k>0$}, \quad \quad N_0(\mathscr{F})= \bigvee_{i \in Ob(\mathcal{C})} \mathscr{F}(i) \]
where $i_0 \rightarrow i_1 \rightarrow \cdots \rightarrow i_k$ denotes all possible length $k$ sequences of composable morphisms in $\mathcal{C}$, and $\mathscr{F}(i_0)$ denotes the value of the functor $\mathscr{F}$ on the initial object $i_0$ of the sequence. The simplicial maps are induced by their counterparts in the nerve $N_\bullet(\mathcal{C})$ of the category $\mathcal{C}$. More precisely, each of the $k+1$ face maps from $N_k(\mathscr{F})$ to $N_{k-1}(\mathscr{F})$ is given by either one of the two maps induced by dropping the terminal morphisms, or by the composition of any of the $k-1$ sequential pairs of morphisms in the length $k$ sequence of composable morphisms. The $k+1$ degeneracy maps from $N_k(\mathscr{F})$ to $N_{k+1}(\mathscr{F})$ are given by the insertion of the identity morphism in the $k+1$ possible spots. 

\medskip
\noindent
A simplicial $G$-spectrum can be described as a contravariant functor from the category of ordered finite sets to $G\mathscr{S}$. As such, the spectrum $N_k(\mathscr{F})$ can be identified with the value of this functor on the set $[k]$ of integers $\{0, \cdots, k \}$ in increasing order. The $k+1$ face and degeneracy maps described above then correspond to the order preserving injective maps $[k-1] \rightarrow [k]$ and the order preserving surjective maps $[k+1] \rightarrow [k]$ respectively. The homotopy colimit of $\mathcal{F}$ is then defined as the geometric realization of the simplicial $G$-spectrum $N_\bullet(\mathscr{F})$. In order to describe this geometric realization, notice that there is a canonical (covariant) functor $\Delta_\bullet$ from the category of finite ordered sets to spaces, whose value on the set $[k]$ is given by the geometric $k$-simplex $\Delta_k$. By identifying the vertices of $\Delta_k$ with the elements of $[k]$, one obtains canonically induced affine maps between the spaces $\Delta_k$ that correspond to face and degeneracy maps. 

\medskip
\begin{defn} \label{hocolim} (Homotopy colimit)

\noindent
The homotopy colimit of a functor $\mathcal{F} : \mathcal{C} \longrightarrow G\mathscr{S}$ is defined as the $G$-spectrum given by the coequalizer (also known as the geometric realization of the simplicial spectrum $N_\bullet(\mathscr{F})$)
\[  \bigvee_{[k] \rightarrow [l]} N_l(\mathscr{F}) \wedge (\Delta_k)_+ \rightrightarrows \bigvee_n N_n(\mathscr{F}) \wedge (\Delta_n)_+ \longrightarrow \hocolim_{\mathcal{C}} \mathscr{F}, \]
where the two maps are given by applying the order preserving map $[k] \rightarrow [l]$ to the two arguments of $N_l(\mathscr{F}) \wedge (\Delta_k)_+$ respectively. 
\end{defn}

\noindent
The categories $\mathcal{C}$ that are relevant to us in this document are very special. These categories have the shape of finite posets $\mathcal{P}$, such that their opposite poset $\overline{\mathcal{P}}$ is an unaugmented finite CW poset \cite{B} (Section 3)\footnote{The posets in \cite{B} are augmented by adding a minimal object for the sake of convenience} . More precisely, given such a poset $\mathcal{P}$, there is a unique regular finite CW complex $X_{\mathcal{P}}$, so that the poset of cells (under inclusion of their closures) is equivalent to $\overline{\mathcal{P}}$. Furthermore, under this equivalence, cells of dimension $k$ in $X_{\mathcal{P}}$ are indexed by elements $p \in \mathcal{P}$ such that the longest nondegenerate ascending chain of morphisms in $\mathcal{P}$ starting with $p$ has length $k$. In fact, any maximal nondegenerate ascending chain starting with $p$ has length $k$. Let us call this integer $k$ the dimension of $p$ and denote that by $|p| = k$. 

\begin{remark} \label{join}
Given two finite posets $\mathcal{P}$ and $\mathcal{Q}$, so that $\overline{\mathcal{P}}$ and $\overline{\mathcal{Q}}$ are unaugmented CW posets, one may define a new finite poset called the join of $\mathcal{P}$ and $\mathcal{Q}$, denoted by $\mathcal{P}$ \textcircled{$\star$} $\mathcal{Q}$, whose opposite poset is also a CW poset
\[ \mbox{$\mathcal{P}$ \textcircled{$\star$} $\mathcal{Q}$} :=  \{\mathcal{P} \cup  \infty \} \times \{ \mathcal{Q} \cup \infty \} - (\infty, \infty) \]
where $\mathcal{P} \cup \infty$ denotes the poset obtained by adding a terminal object $\infty$ to $\mathcal{P}$ (see \cite{M}, Proposition 1.1). The finite regular CW complex corresponding to $\mathcal{P}$ \textcircled{$\star$} $\mathcal{Q}$ is given by the standard (regular) CW structure on the join $X_{\mathcal{P}} \star X_{\mathcal{Q}}$. 
\end{remark}

\medskip
\noindent
Let us now consider the homotopy colimit of a functor $\mathscr{F} : \mathcal{P} \longrightarrow G\mathscr{S}$, where $\overline{\mathcal{P}}$ is an unaugmented CW poset. Definition \ref{hocolim} describes $\hocolim_{\mathcal{P}} \mathscr{F}$ as a coequalizer
\[  \bigvee_{[k] \rightarrow [l]} \bigvee_{i_0 \rightarrow i_1 \rightarrow \cdots \rightarrow i_l} \mathscr{F}(i_0) \wedge (\Delta_k)_+ \rightrightarrows \bigvee_n \bigvee_{i_0 \rightarrow i_1 \rightarrow \cdots \rightarrow i_n} \mathscr{F}(i_0) \wedge (\Delta_n)_+ \longrightarrow \hocolim_{\mathcal{P}} \mathscr{F}, \]
For $k \geq 1$, consider the full subcategory $\mathcal{P}^k \subseteq \mathcal{P}$ of elements $p \in \mathcal{P}$ such that $|p| < k$. We may restrict $\mathscr{F}$ to a functor $\mathscr{F}^k$ on $\mathcal{P}^k$. It is clear that the face and degeneracy maps preserve the simplicial $G$-spectrum $N_\bullet (\mathscr{F}^k)$, which is a levelwise summand in $N_\bullet(\mathscr{F})$. It follows that one has an increasing filtration $F_k(\hocolim_{\mathcal{P}} (\mathscr{F}))$ of $\hocolim_{\mathcal{P}} \mathscr{F}$ defined for $k \geq 1$, and given by 
\[ F_k(\hocolim_{\mathcal{P}} (\mathscr{F})) := \hocolim_{\mathcal{P}^k} \mathscr{F}^k. \]

\medskip
\noindent
Using the same argument as in \cite{B} (Proposition 3.1), we see that $F_{k+1}(\hocolim_{\mathcal{P}} (\mathscr{F}))$ is obtained inductively, starting with $F_1(\hocolim_{\mathcal{P}} (\mathscr{F})) = \bigvee _{p \in  \mathcal{P}, |p|=0} \mathscr{F}(p) $, and constructing  $F_{k+1}(\hocolim_{\mathcal{P}} (\mathscr{F}))$ as a pushout diagram

\[
\xymatrix{\bigvee_{p \in \mathcal{P}, |p|=k} \mathscr{F}(p) \wedge ({\partial X_{\mathcal{P}_p})}_+,
 \ar[d] \ar[r] &  \bigvee_{p \in  \mathcal{P}, |p|=k} \mathscr{F}(p) \wedge ({X_{\mathcal{P}_p})}_+  \ar[d] \\
  F_k(\hocolim_{\mathcal{P}} (\mathscr{F}))  \ar[r] & F_{k+1}(\hocolim_{\mathcal{P}} (\mathscr{F}))
}
\]

\noindent
where $X_{\mathcal{P}_p}$ is the regular CW complex that corresponds to the subcategory $\mathcal{P}_p \subseteq \mathcal{P}$ of objects over $p$ and $\partial X_{\mathcal{P}_p}$ denotes the subcomplex of $X_{\mathcal{P}_p}$ corresponding to the subcategory $\mathcal{P}_p \cap \mathcal{P}^k$. The top horizontal and left vertical maps are the canonical maps induced by the inclusion of the subcategories $(\mathcal{P}_p \cap \mathcal{P}^k) \subseteq \mathcal{P}_p$ and $(\mathcal{P}_p \cap \mathcal{P}^k) \subseteq \mathcal{P}^k$ respectively. In fact, $X_{\mathcal{P}_p}$ is homeomorphic to a closed $k$-disc, and $\partial X_{\mathcal{P}_p}$ is homeomorphic to its boundary sphere of dimension $k-1$. It follows there is a cofiber sequence of the form
\[ F_k(\hocolim_{\mathcal{P}} (\mathscr{F})) \longrightarrow F_{k+1}(\hocolim_{\mathcal{P}} (\mathscr{F})) \longrightarrow \bigvee_{p \in \mathcal{P}, |p|=k} \Sigma^k \mathscr{F}(p). \]

\medskip
\noindent
We package the above observations as the following useful theorem.

\medskip
\begin{thm} \label{CWP} 
Let $\mathcal{P}$ be a finite poset so that $\overline{\mathcal{P}}$ is an unaugmented CW poset. Then given any functor $\mathscr{F} : \mathcal{P} \longrightarrow G\mathscr{S}$, the homotopy colimit $\hocolim_{\mathcal{P}} \mathscr{F}$ admits an increasing filtration 
\[ F_k(\hocolim_{\mathcal{P}} (\mathscr{F})) := \hocolim_{\mathcal{P}^k} \mathscr{F}^k, \quad \mbox{for $k \geq 1$, \, with} \quad F_1(\hocolim_{\mathcal{P}} (\mathscr{F})) = \bigvee _{p \in  \mathcal{P}, |p|=0} \mathscr{F}(p).\]
Furthermore, one may construct  $F_{k+1}(\hocolim_{\mathcal{P}} (\mathscr{F}))$ inductively as a pushout 
\[
\xymatrix{\bigvee_{p \in \mathcal{P}, |p|=k} \mathscr{F}(p) \wedge ({\partial X_{\mathcal{P}_p})}_+,
 \ar[d] \ar[r] &  \bigvee_{p \in \mathcal{P}, |p|=k} \mathscr{F}(p) \wedge ({X_{\mathcal{P}_p})}_+  \ar[d] \\
  F_k(\hocolim_{\mathcal{P}} (\mathscr{F}))  \ar[r] & F_{k+1}(\hocolim_{\mathcal{P}} (\mathscr{F})) }
\]
where $X_{\mathcal{P}_p}$ is the regular CW complex that corresponds to the subcategory $\mathcal{P}_p \subseteq \mathcal{P}$ of objects over $p$, and $\partial X_{\mathcal{P}_p}$ denotes the subcomplex of $X_{\mathcal{P}_p}$ corresponding to the subcategory $\mathcal{P}_p \cap \mathcal{P}^k$. The top horizontal and left vertical maps are the canonical maps induced by the inclusion of the subcategories $(\mathcal{P}_p \cap \mathcal{P}^k) \subseteq \mathcal{P}_p$ and $(\mathcal{P}_p \cap \mathcal{P}^k) \subseteq \mathcal{P}^k$ respectively. By \cite{B} (Proposition 3.1), $X_{\mathcal{P}_p}$ is homeomorphic to the $k$-disc, with $\partial X_{\mathcal{P}_p}$ being its boundary. In particular, there is a cofiber sequence of the form
\[ F_k(\hocolim_{\mathcal{P}} (\mathscr{F})) \longrightarrow F_{k+1}(\hocolim_{\mathcal{P}} (\mathscr{F})) \longrightarrow \bigvee_{p \in \mathcal{P}, |p|=k} \Sigma^k \mathscr{F}(p). \]
\end{thm}

\noindent
Theorem \ref{CWP} reconstructs the homotopy colimit of a functor $\mathscr{F}$ inductively in the same way that the underlying regular CW complex $X_{\mathcal{P}}$ is constructed by attaching the cells $X_{\mathcal{P}_p}$. The only difference is that now one replaces $X_{\mathcal{P}_p}$ with $\mathscr{F}(p) \wedge ({X_{\mathcal{P}_p})}_+$. Notice that if we were to subdivide the cells of $X_{\mathcal{P}}$ so as to refine it to another regular CW complex $X_{\mathcal{R}}$, where $\mathcal{R}$ is a subdivision of the poset $\mathcal{P}$, one would clearly not change the homotopy colimit of $\mathscr{F}$, provided one extends $\mathcal{F}$ to $\mathcal{R}$ compatibly with the subdivision of posets. On the other hand, the filtration on the homotopy colimit would change as a result of this subdivision. We formally explore this process of subdivision in what follows. 

\smallskip
\begin{defn} \label{refine} (Subdivision of posets, compare \cite{St} Section 7)

\noindent
Let $\mathcal{P}$ and $\mathcal{R}$ be finite posets so that $\overline{\mathcal{P}}$ and $\overline{\mathcal{R}}$ are unaugmented CW posets. We say that $\mathcal{R}$ is a subdivision of $\mathcal{P}$ if there exists a surjective map of posets
\[ \pi : \mathcal{R} \longrightarrow \mathcal{P} \]
with the following property. Given $p \in \mathcal{P}$, let $\mathcal{P}_p$ the subposet of objects over $p$, We demand that the regular CW subcomplex of $X_{\mathcal{R}}$ corresponding to $\pi^{-1}(\mathcal{P}_p)$ is a disc of dimension $|p|$ with boundary being precisely those cells indexed by elements $r \in \pi^{-1} (\mathcal{P}_p)$ such that $|\pi(r)| < |p|$. 
\end{defn}

\begin{remark} 
It is not hard to see that given $\mathcal{P}$ and $\mathcal{R}$ as in definition \ref{refine}, $X_{\mathcal{R}}$ is a subdivision of the regular CW complex $X_{\mathcal{P}}$ (as defined in \cite{LW}), which is also regular. Moreover, under this subdivision, the (open) cells of $X_{\mathcal{R}}$ that belong to the interior of the cell indexed by $p \in \mathcal{P}$ are precisely the ones indexed by the set $\pi^{-1}(p) \subseteq \mathcal{R}$. 
\end{remark}

\smallskip
\begin{example} \label{subex}
Consider the following instructive example of a subdivision that is relevant in the proof of theorem \ref{braid rel}. We take $\mathcal{P}_n$ to be the poset representing the standard regular CW decomposition of the $n$-disc, which we will denote by $X_{\mathcal{P}_{n}}$ with one interior $n$-cell, and $2$-hemispherical $k$-cells for all $0 \leq k < n$. We define $\mathcal{R}_{(n+1)}$ to be the poset given by the join $\bullet$  \textcircled{$\star$} $\mathcal{P}_n$ of the one-element poset $\bullet$ and the poset $\mathcal{P}_n$. By remark \ref{join} we see that $X_{\mathcal{R}_{(n+1)}} $ is a cone on $X_{\mathcal{P}_{n}}$ or its topological join with a point, $pt \star X_{\mathcal{P}_{n}}$. One can describe $X_{\mathcal{R}_{(n+1)}}$ as a subdivision of $X_{\mathcal{P}_{(n+1)}}$ which has the following form. For $0 < k \leq n$, one of the two open $k$-cells of $X_{\mathcal{P}_{(n+1)}}$ gets subdivided into two $k$-cells, by introducing one interior $(k-1)$-cell called the separating cell. The closure relations for this subdivision are encapsulated by a map of posets $\pi : \mathcal{R}_{(n+1)} \longrightarrow \mathcal{P}_{(n+1)}$ that satisfies the conditions given in definition \ref{refine}. We describe such a map in definition \ref{PORS}. 
\end{example}

\medskip
\begin{thm} \label{subdiv} 
Let $\pi : \mathcal{R} \longrightarrow \mathcal{P}$ be a be a subdivision as in definition \ref{refine} and assume that we are given a functor $\mathscr{F} : \mathcal{P} \longrightarrow G\mathscr{S}$. Then there is a map of filtered spectra, which is an equivalence on the global homotopy colimit
\[ \iota: \hocolim_{\mathcal{P}} \mathscr{F} \llra{\simeq} \hocolim_{\mathcal{R}} \pi^\ast \mathscr{F}. \]
Furthermore, on the associated quotients, $\iota$ induces the diagonal map for each $p \in \mathcal{P}$ with $|p| = k$
\[ \iota : \Sigma^k \mathscr{F}(p) \llra{\Delta} \bigvee_{r \in \pi^{-1}(p), |r|=k} \Sigma^k \mathscr{F}(p). \]
\end{thm}
\begin{proof}
The inductive description of $\hocolim_{\mathcal{P}} \mathscr{F}$ in theorem \ref{CWP} shows that $\hocolim_{\mathcal{R}} \pi^\ast \mathscr{F}$ is a refinement of $\hocolim_{\mathcal{P}} \mathscr{F}$ and hence there is a map of filtered spectra, which is an equivalence on the global homotopy colimits
\[ \iota: \hocolim_{\mathcal{P}} \mathscr{F} \llra{\simeq} \hocolim_{\mathcal{R}} \pi^\ast \mathscr{F}. \]
Furthermore, since $\mathcal{R}$ is a subdivision of $\mathcal{P}$, on associated graded object, it follows that $\iota$ is given by a sum of diagonal maps for each $p \in \mathcal{P}$ of dimension $k$
\[ \iota : \Sigma^k \mathscr{F}(p) \longrightarrow \bigvee_{r \in \pi^{-1}(p), |r|=k} \Sigma^k \mathscr{F}(p). \]
\end{proof}

\pagestyle{empty}
\bibliographystyle{amsplain}
\providecommand{\bysame}{\leavevmode\hbox
to3em{\hrulefill}\thinspace}

\end{document}